\documentclass[12pt,a4paper,reqno]{amsart}
\usepackage{amsmath,amsfonts,amssymb,amsthm,amscd}
\usepackage{pst-node,pstricks,pst-tree}

\input{comment.sty}
\includecomment{MM}
\includecomment{CL}
\renewcommand{\a}{\alpha}
\renewcommand{\b}{\beta}
\newcommand{\g}{\gamma}
\newcommand{\G}{\Gamma}
\renewcommand{\d}{\delta}

\renewcommand{\l}{\lambda}

\newcommand{\m}{\mu}
\newcommand{\n}{\nu}
\renewcommand{\o}{\omega}
\renewcommand{\O}{\Omega}
\renewcommand{\r}{\rho}
\newcommand{\s}{\sigma}
\renewcommand{\SS}{\Sigma}
\renewcommand{\t}{\tau}
\renewcommand{\th}{\theta}
\newcommand{\e}{\varepsilon}
\newcommand{\f}{\varphi}

\newcommand{\x}{\xi}

\newcommand{\y}{\eta}

\newcommand{\p}{\psi}

\newcommand{\A}{{\mathcal A}}

\newcommand{\R}{{\mathbb R}}

\newcommand{\Z}{{\mathbb Z}}
\renewcommand{\Pr}{{\mathbb P}}
\newcommand{\Hb}{{\mathbb H}}
\renewcommand{\H}{{\mathcal H}}
\newcommand{\Lc}{{\mathcal L}}
\newcommand{\Cc}{{\mathcal C}}

\newcommand{\Uc}{{\mathcal U}}
\newcommand{\Dc}{{\mathcal D}}

\newtheorem{theorem}{Theorem}[section]
\newtheorem{proposition}[theorem]{Proposition}
\newtheorem{lemma}[theorem]{Lemma}
\newtheorem{corollary}[theorem]{Corollary}

\theoremstyle{definition}
\newtheorem{definition}[theorem]{Definition}

\newtheorem{notation}[theorem]{Notation}

\theoremstyle{claim}
\theoremstyle{remark}
\newtheorem{remark}[theorem]{Remark}

\begin{document}

\title[Hausdorff spectrum of harmonic measure]
      {Hausdorff spectrum of harmonic measure}
\author{Ryokichi Tanaka}
\address{Tohoku University, 2-1-1 Katahira, Aoba-ku, 980-8577 Sendai, Japan}
\email{rtanaka@m.tohoku.ac.jp}
\date{\today. Supported by JSPS Grant-in-Aid for Young Scientists (B) 26800029.}

\maketitle

\begin{abstract}
For every non-elementary hyperbolic group, we show that for every random walk with finitely supported admissible step distribution, the associated entropy equals the drift times the logarithmic volume growth if and only if the corresponding harmonic measure is comparable with Hausdorfff measure on the boundary.
Moreover, we introduce one parameter family of probability measures which interpolates a Patterson-Sullivan measure and the harmonic measure, and establish a formula of Hausdorff spectrum (multifractal spectrum) of the harmonic measure.
We also give some finitary versions of dimensional properties of the harmonic measure.
\end{abstract}

\section{Introduction}\label{intro}

Let $\G$ be a finitely generated group.
For a probability measure $\m$ on it, we obtain a random walk on $\G$ by multiplying from right independent random elements with the law $\m$, and the distribution of the random walk at the time $n$ is given by the $n$-th convolution power $\m^{\ast n}$. 
There are several important quantities which capture the asymptotic behaviours of the random walks.
Define the entropy $h$ and the drift $l$ (also called the rate of escape, or the speed) by
$$
h:=\lim_{n \to \infty}-\frac{1}{n}\sum_{x \in \G}\m^{\ast n}(x)\log \m^{\ast n}(x), \ \ \ 
l:=\lim_{n \to \infty}\frac{1}{n}\sum_{x \in \G}|x|\m^{\ast n}(x),
$$
where $|\cdot|$ denotes the word norm associated with a finite symmetric set of generators of $\G$.
It is known that the entropy, introduced by Avez \cite{Av}, is equal to $0$ if and only if all bounded $\m$-harmonic functions on $\G$ are constants (\cite{D}, \cite{KV}).
The entropy and the drift are connected via the logarithmic volume growth $v$ of the group which is defined by
$
e^v:=\lim_{n \to \infty} |B_n|^{1/n},
$
where $|B_n|$ denotes the cardinality of the set $B_n$ of words of length at most $n$, by the inequality
\begin{equation}\label{fundineq}
h \le l v,
\end{equation}
as soon as all those quantities are well-defined (\cite{Gu}, see also e.g.,\ \cite{BHM08} and \cite{V}).
The inequality (\ref{fundineq}) is also called the fundamental inequality.
In \cite{V}, Vershik proposed to study the equality case of (\ref{fundineq}).
In this paper, we focus on hyperbolic groups in the sense of Gromov and characterise the equality of (\ref{fundineq}) in terms of the boundary behaviours of the random walks.
For every hyperbolic group $\G$, one can define the geometric boundary $\partial \G$, which is compact and admits a metric $d_\e$ with a parameter $\e>0$.
The harmonic measure $\n$ is defined by the hitting distribution of the random walk starting from the identity on the boundary $\partial \G$, corresponding to the step distribution $\m$ on $\G$.
The boundary $\partial \G$ has the Hausdorff dimension $D=v/\e$ and the $D$-Hausdorff measure $\H^D$ is finite and positive on $\partial \G$ \cite{C}.
Here the $D$-Hausdorff measure $\H^D$ is a natural measure to compare with the harmonic measure $\n$.
We call a probability measure $\m$ on the group $\G$ admissible if the support of $\m$ generates the whole group $\G$ as a semigroup.
In the present paper, $\m$ is always finitely supported and admissible unless stated otherwise.

\begin{theorem}\label{thm1}
For every finitely supported admissible probability measure $\m$ on every finitely generated non-elementary hyperbolic group $\G$ equipped with a word metric,
it holds that $h=l v$ if and only if the corresponding harmonic measure $\n$ and the $D$-Hausdorff measure $\H^D$ are mutually absolutely continuous and their densities are uniformly bounded from above and from below.
\end{theorem}

In Section \ref{HS}, we prove this result in Theorem \ref{fund}.
See Remark \ref{rigidity}, for another equivalent condition to the equality case.
Blach\`ere, Haissinsky and Mathieu established this result for every finitely supported admissible and symmetric probability measure $\m$ \cite[Corollary 1.2, Theorem 1.5]{BHM11}.
We extend it to non-symmetric measures from a completely different approach as we describe later.
Recently, Gou\"ezel, Math\'eus and Maucourant 
have proven that for a non-elementary hyperbolic group $\G$ which is not virtually free and equipped with a word metric,
for every finitely supported admissible probability measure $\m$, the equality $h=lv$ never holds \cite{GMMpre}.
Together with their results, one concludes that in this setting, the harmonic measure $\n$ and the $D$-Hausdorff measure $\H^D$ are always mutually singular.
Connell and Muchnik proved that for an infinitely supported probability measure $\m$ on $\G$, the $D$-Hausdorff measure $\H^D$ (and also a Patterson-Sullivan measure) and the harmonic measure for $\m$ can be equivalent (\cite[Remark 0.5]{CM1} and \cite{CM2}).
On the other hand, Le Prince showed that for every finitely generated non-elementary hyperbolic group $\G$, there exists a finitely supported admissible and symmetric probability measure $\m$ such that the corresponding harmonic measure $\n$ and the $D$-Hausdorff measure $\H^D$ are mutually singular \cite{LeP}.
Ledrappier proved the corresponding result to Theorem \ref{thm1} for non-cyclic free groups for every finitely supported admissible probability measure $\m$ \cite[Corollary 3.15]{L} .
For free groups, it is straightforward to see that if $\m$ depends only on the word length associated with the standard symmetric generating set, then the corresponding harmonic measure coincides with the Hausdorff measure (the uniform measure on the boundary) up to a multiplicative constant.
Let us also mention some related works.
In \cite{GMM}, Gou\"ezel et al.\ studied a variant of the fundamental inequality (\ref{fundineq}) and also obtained some rigidity results for the equality case.
Apart from Cayley graphs of groups, Lyons extensively studied the equivalence of the harmonic measure and the Patterson-Sullivan measure for universal covering trees of finite graphs \cite{Lyo}.

A novel feature of our approach is to introduce one parameter family of probability measures $\m_\th$, which interpolates a Patterson-Sullivan measure and the harmonic measure on the boundary $\partial \G$.
Let us consider for every $\th \in \R$,
$$
\b(\th):=\lim_{n \to \infty}\frac{1}{n}\log \sum_{x \in S_n}G(1, x)^\th,
$$
where $G(x, y)$ is the Green function associated with $\m$, we denote by $1$ the identity of the group, and by $S_n$ the set of words of length $n$.
The limit exists by the Ancona inequality (Lemma \ref{beta}), and we show that $\b$ is convex, in fact, analytic except for at most finitely many points and continuously differentiable at every point (Proposition \ref{prss} and Corollary \ref{conti-diff}).
Theorem \ref{thm1} is deduced via the following dimensional properties of the harmonic measure $\n$.

\begin{theorem}\label{thm2}
Let $\G$ and $\m$ be as in Theorem \ref{thm1}, and
$\n$ be the corresponding harmonic measure on the boundary $\partial \G$.
It holds that
\begin{equation}\label{loc}
\lim_{r \to 0} \frac{\log \n\left(B(\x, r)\right)}{\log r}=\frac{h}{\e l}, \ \ \ \text{$\n$-a.e.\ $\x$}.
\end{equation}
Define the set 
$$
E_\a=\left\{\x \in \partial \G \ \bigg| \ \lim_{r\to 0}\frac{\log \n\left(B(\x, r)\right)}{\log r}=\a \right\},
$$
then the Hausdorff dimension of the set $E_\a$ is given by the Legendre transform of $\b$, i.e., 
$$
\dim_H E_\a=\frac{\a \th + \b(\th)}{\e},
$$
for every $\a=-\b'(\th)$, 
where $\b$ is continuously differentiable on the whole $\R$, and
$B(\x, r)$ denotes the ball of radius $r$ centred at $\x$.
\end{theorem}

In Section \ref{HS}, we prove this result in Theorem \ref{dim-ent-speed} and Theorem \ref{Haus} in a more detailed form.
The measure $\m_\th$ is constructed by the Patterson-Sullivan technique (Theorem \ref{FGPSmeas}).
As $\th=0$, the measure $\m_0$ is a Patterson-Sullivan measure, and thus comparable with the $D$-Hausdorff measure $\H^D$, and as $\th=1$, the measure $\m_1$ is the harmonic measure $\n$ (Corollary \ref{FGPSth01}).
The probability measure $\m_\th$ satisfies that $\m_\th(E_\a)=1$ for $\a=-\b'(\th)$.
We call $\dim_H E_\a$ as a function in $\a$ the {\it Hausdorff spectrum} of the measure $\n$.
To determine the Hausdorff spectrum is called multifractal analysis which has been extensively studied in fractal geometry.
In fact, there are many technical similarities to analyse harmonic measures and self-conformal measures (\cite{Fe} and \cite{PU}).
For the backgrounds on this topic, see also \cite{F} and references therein.
We show that the measure $\m_\th$ satisfies a Gibbs-like property with respect to $\b(\th)$, where $\b(\th)$ is an analogue of the pressure (Corollary \ref{Gibbsth}).
This measure $\m_\th$ is also characterised by the eigenmeasures of certain transfer operator built on a symbolic dynamical system associated with an automatic structure of the group (Theorem \ref{am=mth}).
To study the measure $\m_\th$, we employ the results about the Martin boundary of a hyperbolic group by Izumi, Neshveyev and Okayasu \cite{INO} and a generalised thermodynamic formalism due to Gou\"ezel \cite{G}.
Note that the formula (\ref{loc}) is proved for every non-elementary hyperbolic group and for every finitely supported symmetric probability measure $\m$ in \cite[Theorem 1.3]{BHM11}, 
for every non-cyclic free groups and for every probability measure $\m$ of finite first moment in \cite[Theorem 4.15]{L},
for a general class of random walks on trees in \cite{Ktree} and for the simple random walks on the Galton-Watson trees in \cite{LPP}.

The following result is a finitary version of Theorem \ref{thm2}, inspired by the corresponding results for the Galton-Watson trees by Lyons, Pemantle and Peres \cite{LPP}.

\begin{theorem}\label{thm3}
Let $\G$ and $\m$ be as in Theorem \ref{thm1}, and consider the associated random walk starting at the identity on $\G$.
For every $a \in (0,1)$, 
there exists a subset $\G_a \subset \G$ such that the random walk stays in $\G_a$ for every time with probability at least $1-a$, and 
$$
\lim_{n \to \infty}|\G_a \cap S_n|^{1/n}=e^{h/l}.
$$
\end{theorem}

In Section \ref{FR}, we prove this result in Theorem \ref{confinement}.
In particular, if $h<l v$, then the random walk is confined in an exponentially small part of the group with positive probability.
This can be compared with \cite[p.669]{V}, where a random generation of group elements which is called the Monte Carlo method is discussed.
For example, the random generation of group elements according to a random walk does not produce the whole data of the group in this case;
see also \cite{GMMpre}.

Let us return to Theorem \ref{thm1}.
For a symmetric probability measure $\m$, i.e., $\m(x)=\m(x^{-1})$ for every $x \in \G$,
one can define the Green metric $d_G(x, y)=-\log F(x, y)$, where $F(x, y)$ denotes the probability that the random walk starting at $x$ ever reaches $y$, and show that $\G$ is hyperbolic with respect to $d_G$ according to the Ancona inequality \cite{BHM11}.
The Green metric $d_G$ is not geodesic; nevertheless, one can use approximate trees argument and most of common techniques for the geodesic case work.
The harmonic measure $\n$ is actually a quasi-conformal (in fact, conformal) measure with respect to the metric induced in the boundary $\partial \G$ by the Green metric $d_G$, and this fact plays an essential role to deduce Theorem \ref{thm1}
and that the local dimension of the harmonic measure $\n$ is $h/(\e l)$ for $\n$-a.e.\ in Theorem \ref{thm2} in the symmetric case.
The symmetry of $\m$ is required to make the Green metric $d_G$ a genuine metric in $\G$; 
otherwise it is not clear that one can discuss about the hyperbolicity for a non-symmetric metric.
Here our alternative is to introduce a measure $\m_\th$, whose construction is actually inspired by a recent work of Gou\"ezel on the local limit theorem on a hyperbolic group \cite{G}, and to obtain the Hausdorff spectrum of the harmonic measure $\n$.
In many cases, a description of the Hausdorff spectrum is a main purpose on its own right, especially, when it is motivated by a problem in statistical physics.
A fundamental observation in this paper, however, is rather the converse; namely, we use the multifractal analysis to compare the harmonic measure with a natural reference measure which is the $D$-Hausdorff measure on the boundary of the group $\G$.
More precisely, we shall see that the function $\b$ is affine on $\R$ if and only if those two measures are mutually absolutely continuous (Theorem \ref{fund}).
Furthermore, the description of the Hausdorff spectrum implies that the harmonic measure has a rich multifractal structure as soon as it is singular with respect to the $D$-Hasudorff measure (Theorem \ref{Haus}).
In particular, the range of the Hausdorff spectrum contains the interval $[h/(\e l), v/\e]$.

Let us mention about an extension to a step distribution $\m$ of unbounded support.
The arguments in the present paper work once we have the Ancona inequality and its strengthened one (Theorem \ref{Holder}).
Gou\"ezel has proven those for every admissible probability measure $\m$ with a super-exponential tail \cite{Gpre}. 
At the same time, he has also proven a failure of description of the Martin boundary in the usual sense for an admissible probability measure with an exponential tail [ibid].
Hence the results in this paper are extended to every admissible step distribution $\m$ with a super-exponential tail, 
but it is obscure whether one could extend to $\m$ with an exponential tail in the present approach.
We shall also mention about an extension to a left invariant metric which is not induced by a word length in $\G$.
For example, one is interested in the setting where $\G$ acts cocompactly on the hyperbolic space $\Hb^n$ and the metric in $\G$ is defined by $d(x, y):=d_{\Hb^n}(xo, yo)$ for a reference point $o$ in $\Hb^n$.
In fact, in \cite{BHM11}, they proved Theorem \ref{thm1} for symmetric $\m$ with every metric $d$ which is hyperbolic and quasi-isometric to a word metric in $\G$ (not necessarily geodesic).
Some of our results still hold for such a metric $d$, and in fact, most of the results are expected to remain valid; but we do not proceed to this direction in the present paper for the simplification of the proofs.

Finally, we close this introduction by pointing out some related problems in a continuous setting.
On the special linear groups, the regularity problem of the harmonic measures is proposed by Kaimanovich and Le Prince \cite{KLeP}, and they showed that there exists a finitely supported symmetric probability measure (the support can generate a given Zariski dense subgroup) on $SL(d, \R)$ ($d \ge 2$) such that the corresponding harmonic measure is non-atomic singular with respect to a natural smooth measure class on the Furstenberg boundary.
They proved this result via the dimension inequality of the harmonic measure.
Bourgain constructed a finitely supported symmetric probability measure on $SL(2, \R)$ such that the corresponding harmonic measure is absolutely continuous with respect to Lebesgue measure on the circle \cite{Bou}.
Brieussel and the author proved analogous results in the three dimensional solvable Lie group $Sol$ \cite{BT}, namely, they showed that random walks with finitely supported step distributions on it can produce both absolutely continuous and singular harmonic measures with respect to Lebesgue measure on the corresponding boundary.
The dimension inequality is also used there to prove the existence of a finitely supported probability measure whose harmonic measure is singular. 
We point out, however, the dimension equality of type (\ref{loc}) is still missing in both cases, and in general, in the Lie group settings to the extent of our knowledge.

The organisation of this paper is the following:
In Section \ref{P}, we recall some known results about hyperbolic groups, random walks on it, and the Martin boundary.
In Section \ref{HM}, we construct the measure $\m_\th$, and show several properties of $\m_\th$ including the Gibbs property (Corollary \ref{Gibbsth}) and the ergodicity (Corollary \ref{ergodic}).
In Section \ref{Sym} we review a generalised thermodynamic formalism associated with an automatic structure of the group due to Gou\"ezel \cite{G} and prove the regularity of $\b$.
In Section \ref{HS}, we prove Theorem \ref{thm2} in Theorem \ref{dim-ent-speed} and in Theorem \ref{Haus}, and then we deduce Theorem \ref{thm1} in Theorem \ref{fund}.
In Section \ref{FR}, we show finitary results, following the Galton-Watson trees cases \cite{LPP}, and prove Theorem \ref{thm3} in Theorem \ref{confinement}.

\begin{notation}
Throughout this article, we use $C, C', \dots$, to denote absolute constants whose exact values may change from line to line,
and also use them with subscripts, for instance, $C_\d$ to specify its dependency only on $\d$.
For functions $f$ and $g$, we write $f \asymp g$ if there exist some constants $c_1, c_2 >0$ such that $c_1 g \le f \le c_2 g$, and also write $f \asymp_{\d} g$ if the implied constants depend only on $\d$.
Furthermore, for measures $\n_1$ and $\n_2$, we say that those are comparable if $\n_1$ and $\n_2$ are mutually absolutely continuous and their densities are bounded from above and from below, and also write $\n_1 \asymp \n_2$.
We denote by $f(n)=O(g(n))$  if there exists a constant $C>0$ independent of $n$ such that $f(n) \le Cg(n)$,
by $f(n)=o(g(n))$ if $f(n)/g(n) \to 0$ as $n \to \infty$, and by $f(n)\sim g(n)$ if $f(n)/g(n) \to 1$ as $n \to \infty$.
\end{notation}

\section{Preliminary}\label{P}

\subsection{Hyperbolic groups}
Let $\G$ be a finitely generated group with a finite symmetric set $S$ of generators.
We denote by $|x|$ the word length of $x$ in $\G$ associated with $S$, and by $d(x,y)=|x^{-1}y|$ the corresponding word metric which is left invariant under the action of $\G$. 
(We do not specify $S$.)
The Gromov product is defined by
$$
(x|y)_z=\frac{d(x,z)+d(y,z)-d(x,y)}{2},
$$
for $x$, $y$, and $z$ in $\G$.
We say that the metric space $(\G, d)$ is $\d$-hyperbolic, if there exists a constant $\d \ge 0$ such that for all $x, y, z$ and $w$ in $\G$,
it holds that 
$$
(x|y)_w \ge \min \left((x|z)_w, (y|z)_w\right)-\d.
$$
If $(\G, d)$ is $\d$-hyperbolic, then every geodesic triangle is $4\d$-slim, i.e., 
each side of the triangle is within the $4\d$-neighbourhood of the union of the other two sides.
The notion of the hyperbolicity, which means that it is $\d$-hyperbolic for some $\d\ge 0$, is invariant under quasi-isometry between two geodesic metric spaces, 
hence it does not depend on the choice of the finite symmetric set of generators.
More details about the hyperbolic groups, see \cite{GdlH}.

The boundary $\partial \G$ of $\G$ is defined 
in the following way:
Fix a base point the identity $1$ of $\G$ and write $(x|y)$ for $(x|y)_1$.
We say that a sequence $\{x_n\}_{n=0}^{\infty}$ converges to infinity if $(x_n|x_m) \to \infty$ as $n$ and $m$ tend to infinity, and
two sequences $\{x_n\}_{n=0}^{\infty}$ and $\{y_n\}_{n=0}^{\infty}$ converging to infinity are equivalent if $(x_n|y_m) \to \infty$ as $n$ and $m$ tend to infinity. 
We define the geometric boundary $\partial \G$ as the set of sequences converging to infinity modulo the equivalence relation.
Since $(\G, d)$ is a proper geodesic space (i.e., all closed balls are compact), 
the boundary $\partial \G$ is also identified with the set of geodesic rays modulo bounded distances \cite[Chapitre 7, Proposition 4]{GdlH}.
The Gromov product is extended to $\G \cup \partial \G$ by setting
$$
(\x|\y)=\sup \left\{\liminf_{n, m \to \infty} (x_n|y_m) \ \big| \ \{x_n\} \in \x, \{y_n\} \in \y \right\},
$$
where the supremum is taken over all sequences $\{x_n\}_{n=0}^{\infty}$ converging to $\x$ and $\{y_n\}_{n=0}^{\infty}$ converging to $\y$.
For the extended Gromov product, it holds that
$$
\liminf_{n, m \to \infty} (x_n|y_m) \le (\x|\y) \le \liminf_{n, m \to \infty} (x_n|y_m) + 2\d,
$$
for every such sequence by the $\d$-hyperbolicity.
The space $\G \cup \partial \G$ is compact, equipped with the topology by bases of balls in $\G$ and sets of the form 
$U_r(x)=\{\x \in \G \cup \partial \G \ \big | \ (\x|x) > r\}$.
In fact, it admits a compatible metric $d_\e$ with a parameter $\e>0$ small enough such that 
$$
C^{-1}e^{-\e(\x|\y)} \le d_\e(\x, \y) \le Ce^{-\e(\x|\y)},
$$
for some constant $C > 0$. 
We fix a parameter $\e$ once and for all so that all the quantities associated with the metric in the boundary $\partial \G$ such as the Hausdorff dimension are defined with respect to the metric $d_\e$.

Recall that the shadow $S(x, R)$ for $x$ in $\G$ and $R\ge 0$ is the set of points $\x$ in the boundary $\partial \G$ such that every geodesic ray issuing from $1$ converging to $\x$ passes through the ball of radius $R$ centred at $x$.
The shadows are compared with balls.

\begin{lemma}\label{shadow}
For every $R \ge 4\d$ and for every $\x$ in $\partial \G$, it holds that for every point $x$ on every geodesic ray starting from $1$ converging to $\x$,
$$
B(\x, C^{-1}e^{\e R}e^{-\e|x|}) \subset S(x, R) \subset B(\x, Ce^{\e R}e^{-\e|x|}).
$$
\end{lemma}
One is able to show the lemma by the $\d$-hyperbolicity (e.g.\ \cite{Ca}). 
The shadows are used to describe measures on the boundary $\partial \G$.

\subsection{Random walks on groups}

For a countable group $\G$, consider a probability measure $\m$ on it. Here we always assume that the measure $\m$ is admissible, i.e., the support of $\m$ generates the whole group $\G$ as a semigroup.
The random walk $\{x_n\}_{n=0}^{\infty}$ on $\G$ starting at the identity $1$ is defined by $x_n:=X_1\cdots X_n$ and $x_0=1$, where $X_1, X_2, \dots$, is a sequence of independent random elements with the common distribution $\m$.
We denote by $\O=\G^{\Z_{+}}$ the space of paths and by $\Pr$ the distribution of the random walk on $\O$ given by the push-forward measure of the product measure $\m^{\Z_{+}}$ induced by the map $\{X_n\}_{n=1}^{\infty} \mapsto \{x_n\}_{n=0}^{\infty}$.
The distribution of $x_n$ is given by the $n$-th convolution power of $\m$, namely,
$\Pr\left(x_n=x\right)=\m^{\ast n}(x)$.
Let us define the Green function by
$$
G(x, y)=\sum_{n=0}^{\infty}p_n(x,y),
$$
where $p_n(x,y)=\Pr(x_n=x^{-1}y)$.
Recall that the Green function has finite values if and only if the random walk is transient, and it is the case as soon as $\G$ is  a non-amenable group.
Henceforth, we assume that $\G$ is a non-elementary hyperbolic group, i.e., non-amenable, and $\m$ is a finitely supported admissible probability measure on $\G$.
In general, we do not assume that $\m$ is symmetric, i.e., $\m(x)=\m(x^{-1})$.
The following is a special case of the theorem by Kaimanovich \cite[Theorem 7.4]{K}.
We give a sketch of the proof for the sake of convenience.

\begin{theorem}[Kaimanovich]\label{Kaim}
The random walk $\{x_n\}_{n=0}^{\infty}$ converges to a point in the boundary $\partial \G$ almost surely in $\Pr$.
Moreover, there exists a positive constant $l>0$ such that
$$
d(1, x_n) = l n +o(n),
$$
$\Pr$-almost surely, and for $\Pr$-almost every sample path, there exists a unit speed geodesic ray $\g_\o$ emanating from the identity $1$ such that
$$
d(x_n, \g_\o(\lfloor l n\rfloor))=o(n).
$$
Here $\lfloor a \rfloor$ denotes the integer part of $a$.
Moreover, one can take the map from $\O$ to the space of geodesic rays endowed with the pointwise convergence topology: $\o \mapsto \g_\o$ as a Borel measurable map. 
\end{theorem}

\begin{proof}[A sketch of the proof]
For an admissible probability measure $\m$, the drift $l$ is positive since the Poisson boundary of $(\G, \m)$ is non-trivial.
The Kingman subadditive ergodic theorem implies that $d(1, x_n)=l n + o(n)$ for $\Pr$-a.s.
Together with $d(x_n, x_{n+1})=o(n)$ (in fact, bounded in our case), the Gromov product grows lineally: $(x_n|x_{n+1})=l n + o(n)$, a.s.\ in $\Pr$.
It follows that the sequence $\{x_n\}_{n=0}^{\infty}$ is Cauchy with respect to the metric $d_\e(\x, \y)\le Ce^{-\e(\x|\y)}$ in $\G\cup \partial \G$, and thus it converges to a point $\x$ in the boundary.
For a unit speed geodesic ray $\g_\o$ converging to $\x$ from $1$, one can deduce that each $x_n$ is within the distance $o(n)$ from $\g_\o(\lfloor l n\rfloor)$.
The map $\o \mapsto \g_\o$ is taken to be Borel measurable by choosing the lexicographically minimum geodesic ray.
More details, see Section 7 in \cite{K}.
\end{proof}

The harmonic measure is the hitting distribution of the random walk on the boundary.

\begin{definition}
We call the distribution of the point to which the random walk $\{x_n\}_{n=0}^{\infty}$ converges, the {\it harmonic measure} $\n$ corresponding to $\m$.
\end{definition}

The harmonic measure $\n$ is a unique stationary measure on the boundary $\partial \G$, i.e.,
\begin{equation}\label{stationary}
\n=\sum_{g \in \G}\m(g)g_\ast \n,
\end{equation}
where $x_\ast \n$ is the push-forward of $\n$ by the left multiplication of $x$, or equivalently the hitting distribution of the random walk starting from $x$ with the same step distribution $\m$; see e.g.\ \cite[Theorem 2.4]{K}.

\subsection{Martin boundary}

The Green function is factorized as 
$$
G(x, y)=F(x, y)G(y, y),
$$
where $F(x, y)$ is the probability that the random walk starting at $x$ ever hits $y$, and $G(y,y)=G(1,1)$ by the left invariance by $\G$ of the transition probability. 
Associated with the Green function $G(x,y)$, the Martin kernel is defined by $K(x, y):=G(x,y)/G(1,y)$.
The Martin compactification $\overline \G$ of $\G$ is the smallest compactification of $\G$ where $\G$ is equipped with the discrete topology, such that for every $x$, the Martin kernel $K(x, \cdot)$ extends continuously on $\overline \G$.
The Martin boundary is the part which is newly added to $\G$, i.e., $\partial_M \G=\overline \G \setminus \G$.
For every finitely supported admissible $\m$ on $\G$, the Martin boundary $\partial_M \G$ is identified with the geometric boundary $\partial \G$ by Ancona \cite{A} (see also \cite{INO}).
The Poisson boundary $(\partial \G, \n)$ is realised in the Martin boundary, and it holds that for $\n$-a.e.\ $\x$,
\begin{equation}\label{Poisson}
\frac{dx_\ast \n}{d\n}(\x)=K(x, \x).
\end{equation}
More details, see, e.g.\ \cite[Chapter IV]{W}.
Let us recall the so-called Ancona inequality which plays an important role to identify the Martin boundary with $\partial \G$.

\begin{theorem}[\cite{A}, Theorem 3.3, and \cite{INO}, Proposition 2.1]\label{Ancona}
There exists a constant $C\ge 1$ such that for every geodesic segment connecting $x$ and $y$, and every point $z$ on it,
\begin{equation}\label{A}
C^{-1}G(x, z)G(z, y) \le G(x, y) \le C G(x, z)G(z, y). 
\end{equation}
\end{theorem}

For an extension to probability measures $\m$ of super-exponential tails, see \cite{Gpre}.
Note that the first inequality is clear since $F(x, y)\ge F(x, z)F(z, y)$.
The Ancona inequality implies that the Green function $G(x, y)$ is almost multiplicative along geodesics.
We also frequently use the Harnack inequality: there exists a constant $C' \ge 1$ such that for $x, y, z$, and $z'$ in $\G$,
\begin{equation}\label{H}
G(x, z) \le C'^{d(z, z')}G(x, z')
\ \text{ and } \ 
G(z, y) \le C'^{d(z, z')}G(z', y).
\end{equation}
These follow since the Green function is a super-harmonic function.
Combining (\ref{A}) with (\ref{H}), we deduce that for every geodesic segment connecting $x$ and $y$, and every point $z$ within the distance $r$ from the geodesic segment, 
$$
G(x, y) \le CC'^{2r}G(x, z)G(z, y).
$$ 
One of important consequences of (\ref{A}) is that the Martin kernel $K(x, \cdot)$ is H\"older continuous on $\partial \G$ for each $x$ \cite[Theorem 3.3]{INO}.
In this paper, we use a bit stronger form of this fact.

\begin{theorem}\label{Holder}
There exist constant $C>0$ and $\r < 1$ such that if two geodesic segments connecting $x$ and $y$, and $x'$ and $y'$, respectively, have a common geodesic segment of length $n$ within each of those $4\d$-neighbourhoods,
then
$$
\left|\frac{G(x, y)/G(x', y)}{G(x, y')/G(x',y')}-1\right| \le C\r^n.
$$
\end{theorem}

The proof follows from the same argument as in \cite[Theorem 4.6]{GL}. 
See also an extension to probability measures $\m$ of super-exponential tails \cite{Gpre}.
We note that the result is deduced once the Ancona inequality (in fact, a relative version of it given in \cite{INO}) holds.
In \cite{GL} and \cite[Theorem 2.9]{G}, they proved the above estimate for a symmetric $\m$ to deal with the $r$-Green function $G_r(x, y)$. Here we discuss only $G(x, y)$ which is enough for our purpose.
From Theorem \ref{Holder}, for every $x$ and $x'=1$, and every sequence $\{y_n\}_{n=0}^\infty$ along a geodesic ray from $1$ converging to $\x$, the Martin kernel $G(x, y_n)/G(1, y_n)$ converges to the limit $K(x, \x)$. Furthermore, $\log K(x, \x)$ is H\"older continuous in $\x$ with an exponent independent of $x$.

In \cite{INO}, Izumi, Neshveyev and Okayasu proved a Gibbs-like property of a harmonic measure in terms of $G(x, y)$ or $F(x, y)$ in the following sense:

\begin{theorem}[\cite{INO}, Theorem 4.1]\label{Gibbs}
Let $R \ge 4\d$.
It holds that for every $x$ in $\G$,
$$
\n(S(x, R)) \asymp_R G(1, x),
$$
where the implied constants are independent of $x$.
\end{theorem}

\begin{remark}
Actually, they use the following notation in \cite{INO}: 
Denote by $U(\x, m)$ the set of all $\y$ in $\partial \G$ such that for arbitrary two geodesic rays $\g$ from $1$ converging to $\x$ and $\g'$ from $1$ converging to $\y$, respectively, one has $\lim_{n \to \infty}(\g(n)|\g'(n))>m$, where the sequence is nondecreasing, hence the limit exists.
They have shown that for every $\x$ in $\partial \G$ and every geodesic ray from $1$ towards $\x$, and for all $m\ge 0$, 
it holds that $\n(U(\x, m))\asymp F(1, \g(m))$.
The above statement is deduced by the relation: $U(\x, n-R+4\d) \subset S(\g(n), R) \subset U(\x, n-R)$, which follows by the $\d$-hyperbolicity, and by the Harnack inequality (\ref{H}).
\end{remark}

\section{Measure $\m_\th$}\label{HM}

First, we recall the Busemann function on $(\G, d)$.
For each $\x$ in $\partial \G$, take a geodesic ray $\g$ converging to $\x$.
Let
$$
b_\g (x,y):=\lim_{n \to \infty}\left(d(x, \g(n))-d(\g(n), y)\right),
$$
for $x, y$ in $\G$, where the limit exists since a sequence $d(x, \g(n))-n$ is decreasing by the triangular inequality.
The Busemann function of the point $\x$ is 
$$
b_\x(x,y)=\sup\left\{b_\g(x,y) \ | \ \text{geodesic rays $\g$ towards $\x$}\right\}.
$$
By the $\d$-hyperbolicity, it holds that
\begin{equation}\label{Busemann}
b_\g(x, y) \le b_\x(x,y) \le b_\g(x,y)+32\d,
\end{equation}
for every geodesic ray $\g$ towards $\x$, regardless of the starting point \cite[Chapitre 8]{GdlH}.

The function $\b$ defined in the following lemma will play an important role in this paper.

\begin{lemma}\label{beta}
For every $\th \in \R$, the limit 
$$
\b(\th):=\lim_{n\to \infty}\frac{1}{n}\log \sum_{x \in S_n}G(1, x)^\th
$$
exists, where $S_n$ is the set of words of length $n$ in $\G$.
Moreover, the function $\b$ is convex on $\R$. 
\end{lemma}

\begin{proof}
Let $G_\th(n):=\sum_{x \in S_n}G(1,x)^\th$. 
For every $\th \in \R$, by the Ancona inequality (\ref{A}), $G_\th(n)$ is submultiplicative up to a bounded constant; hence the limit of $(1/n)\log G_\th(n)$ exists. 
The convexity follows from
$G_{(1-t)\th_0+t\th_1}(n) \le G_{(1-t)\th_0}(n)G_{t\th_1}(n)$ for $0\le t \le 1$ and $\th_0, \th_1 \in \R$. 
\end{proof}

The function $\b$ is convex on $\R$, therefore it is continuous, and differentiable except for at most countably many points.
In fact, we will show that $\b$ is analytic except for at most finitely many points and continuously differentiable at every point (Proposition \ref{prss} and Corollary \ref{conti-diff}).
This is shown by identifying $\b(\th)$ with the pressure for certain Ruelle-Perron-Frobenius transfer operator, based on an automatic structure of the hyperbolic group $\G$ in Section \ref{Sym}.

Here we note that $\b$ has special values at $\th=0$ and $1$.

\begin{lemma}\label{th01}
It holds that $\b(0)=v$, i.e., the logarithmic volume growth of $(\G,d)$ and $\b(1)=0$.
Furthermore, $\b(\th)\to -\infty$ as $\th \to +\infty$, and $\b(\th) \to +\infty$ as $\th \to -\infty$.
\end{lemma}

\begin{proof}
First, it follows that $\b(0)$ is just the logarithmic volume growth $v>0$ by definition.
Next, we show that $\b(1)=0$.
Note that the shadows $S(x, R)$ $(x \in S_n)$ cover the boundary $\partial \G$ for some fixed $R > 4\d$ with bounded overlaps, i.e., there is $N$ such that every $\x$ in the boundary is included in at most $N$ shadows of the form $S(x, R)$ with $|x|=n$ \cite[Lemme 6.5]{C}.
By Theorem \ref{Gibbs}, we have that 
$$
1=\n(\partial \G) \le \sum_{x \in S_n}\n(S(x, R)) \le C\sum_{x \in S_n}G(1, x).
$$
On the other hand, 
$$
\sum_{x \in S_n}G(1, x) \le C\sum_{x \in S_n}\n(S(x, R)) \le CN\n\left(\cup_{x \in S_n}S(x, R)\right) \le CN.
$$
Therefore, we obtain $\b(1)=0$.

Finally, let us show that $\b(\th) \to -\infty$ (resp.$+\infty$) as $\th \to +\infty$ (resp.$-\infty$).
For the random walk $(\G, \m)$, the spectral radius $\r$ for $\m$ is strictly less than $1$
and
the $n$-th step transition probability satisfies that $p_n(x, y) \le \r^n$. 
For a finite range random walk, it takes at least $d(x, y)/r$ steps to reach $y$ from $x$ for some constant $r>0$.
Therefore, the Green function is estimated by $G(x, y) \le Ce^{-\l d(x, y)}$ for constants $C>0$ and $\l > 0$.
For every positive $\th$, one has $\b(\th) \le v -\th \l$, hence $\b(\th)$ goes to $-\infty$ as $\th$ tends to $+\infty$.
On the other hand, for every negative $\th$, we have the converse inequality: $\b(\th) \ge v -\th \l$, and thus $\b(\th)$ goes to $+\infty$ as $\th$ tends to $-\infty$.
\end{proof}

We introduce one parameter family of measures $\m_\th$ on the boundary $\partial \G$; it interpolates the harmonic measure and a Patterson-Sullivan measure.
Recall that a Patterson-Sullivan measure $\m_0$ is given as a limit point of 
\begin{equation*}
\m_{s}:=\frac{\sum_{x \in \G}e^{-s|x|}\d_x}{\sum_{x \in \G}e^{-s|x|}}
\end{equation*}
as $s \downarrow v$
on the compactification $\G\cup \partial \G$.
The measure $\m_0$ is supported on the boundary $\partial \G$, and satisfies
\begin{equation*}
\frac{dg_\ast \m_0}{d\m_0}(\x) \asymp e^{-vb_\x(g, 1)}
\end{equation*}
for every $g$ in $\G$ and $\m_0$-almost every $\x$ in $\partial \G$.
One can show that all limit points are comparable with the $D$-Hausdorff measure; and thus they form a unique measure class.
This construction corresponds to the case where $\th=0$ in the following.

\begin{theorem}\label{FGPSmeas}
For every $\th \in \R$, there exists a probability measure $\m_\th$ on $\partial \G$ such that
\begin{equation}\label{RN}
\frac{dg_\ast \m_\th}{d\m_\th}(\x) \asymp K(g, \x)^\th e^{-\b(\th)b_\x(g, 1)}
\end{equation}
for every $g$ in $\G$ and $\m_\th$-almost every $\x$ in $\partial \G$.
\end{theorem}

\begin{proof}
The construction is analogous to the one of a Patterson-Sullivan measure on $\partial \G$ \cite[Th\'eor\`eme 5.4]{C}.
Notice that $\b(\th)$ is the exponent of convergence for 
$$
\sum_{x \in \G}G(1, x)^\th e^{-s|x|},
$$
i.e., it converges when $s > \b(\th)$ and diverges when $s < \b(\th)$.
Define the measure on $\G\cup \partial \G$ by
\begin{equation}\label{FGPSs}
\m_{\th, s}:=\frac{\sum_{x \in \G}G(1, x)^\th e^{-s|x|}\d_x}{\sum_{x \in \G}G(1, x)^\th e^{-s|x|}},
\end{equation}
where $\d_x$ denotes the Dirac measure at $x$.
First, suppose that $\sum_{x \in \G}G(1, x)^\th e^{-s|x|}$ diverges as $s \downarrow \b(\th)$.
In this case, since $\m_{\th, s}$ is supported on the compact space $\G \cup \partial \G$, as $s \downarrow \b(\th)$, there is a subsequencial weak limit which we denote by $\m_\th$. The limiting measure $\m_\th$ is supported on $\partial \G$ as 
the denominator of (\ref{FGPSs}) diverges.
For the left multiplication by $g$, the measure $\m_{\th, s}$ has the form:
\begin{align*}
g_\ast \m_{\th,s}(x)	&=\frac{G(1, g^{-1}x)^\th e^{-s|g^{-1}x|}}{\sum_{x \in \G}G(1, x)^\th e^{-s|x|}} \\
					&=\frac{G(1, g^{-1}x)^\th}{G(1, x)^\th}e^{-s(|g^{-1}x|-|x|)}\m_{\th, s}(x).
\end{align*}
Here $G(g, x)/G(1, x)$ goes to $K(g, \x)$ as $x$ tends to $\x$, and $|g^{-1}x|-|x|$ equals to $b_\x(g, 1)$ up to a constant depending only on $\d$ for every $x$ in a neighbourhood of $\x$ in $\G\cup \partial \G$ \cite[Lemme 2.2]{C}.
Then the claim follows.

Next, we consider the case where $\sum_{x \in \G}G(1, x)^\th e^{-s|x|}$ does not diverge as $s \downarrow \b(\th)$.
In this case, one can modify the series a little bit to make it diverge at the exponent of convergence \cite[Lemma 3.1]{P}.
There exists a positive increasing function $k$ on $[0, \infty]$ such that for any $\e>0$ there is $b_0$ satisfying that 
$k(ab)\le a^\e k(b)$ for $b > b_0$ and $a>1$, and 
$\sum_{x \in \G}k(e^{|x|})G(1, x)^\th e^{-s|x|}$ diverges as $s \downarrow \b(\th)$.
We define the measure $\m_{\th, s}^k$ with the modification by $k$ as in (\ref{FGPSs}). 
The rest follows in a similar way to the above.
\end{proof}

\begin{corollary}\label{FGPSth01}
For $\th=0$, a measure $\m_0$ is comparable with the $D$-Hausdorff measure $\H^{D}$ on $\partial \G$, 
where $D=v/\e$,
i.e., 
$
\m_0 \asymp \H^D.
$
For $\th=1$, a measure $\m_1$ coincides with the harmonic measure $\n$, i.e., $\m_1=\n$.
\end{corollary}

\begin{proof}
When $\th=0$, a measure $\m_0$ is nothing but a Patterson-Sullivan measure as we see in the construction, and thus $\m_0$ is comparable with the $D$-Hausdorff measure (\cite[Corollaire 7.5]{C}; see also \cite[Corollary 2.5.10]{Ca}).
When $\th=1$, in the proof of Theorem \ref{FGPSmeas}, in fact we have 
$$
\frac{dg_\ast \m_1}{d\m_1}(\x)=K(g, \x),
$$
since a parameter $s$ converges from above to $\b(1)$ which is equal to $0$ (Lemma \ref{th01}).
Recall that the harmonic measure $\n$ is the unique stationary measure on the boundary $\partial \G$ satisfying (\ref{stationary}).
By taking the Radon-Nikodym derivative of $g_\ast \n$ with respect to $\n$, noting (\ref{Poisson}), we have 
$
1=\sum_{g \in \G}\m(g)K(g, \x),
$
for $\n$-a.e. $\x$. By the continuity of the Martin kernel $K(g, \x)$ in $\x$, we deduce that a probability measure $\m_1$ is $\m$-stationary on $\partial \G$, hence it has to coincide with the harmonic measure $\n$.
\end{proof}

\begin{corollary}[A Gibbs property of $\m_\th$]\label{Gibbsth}
For every $\th \in \R$,
let $\m_\th$ be a probability measure satisfying (\ref{RN}) in Theorem \ref{FGPSmeas}: Then there exists $R_0$ such that for every $R \ge R_0$ and for every $g$ in $\G$,
$$
\m_\th\left(S(g, R)\right) \asymp_{\th, R} G(1, g)^\th e^{-\b(\th)|g|}.
$$
\end{corollary}

\begin{proof}
For every shadow $S(g, R)$, and for every $\x \in S(g, R)$, take an arbitrary geodesic ray $\g$ from $1$ converging to $\x$.
By the Ancona inequality (\ref{A}), for $|g|<m$,
we have
$$
G(1, \g(m)) \asymp G(1, \g(|g|))G(\g(|g|), \g(m)).
$$
Since $d(g, \g(|g|)) \le 2R$, by the Harnack inequality (\ref{H}), it follows that 
$G(1, \g(m)) \asymp_R G(1, g)G(g, \g(m))$.
Hence by letting $\g(m) \to \x$, 
we have
$G(1, g)K(g, \x) \asymp_R 1$.
On the other hand, 
one obtains $b_\x(g, 1)=-|g|+O(R, \d)$ for all $\x \in S(g, R)$ by the $\d$-hyperbolicity and by (\ref{Busemann}).
Since $\m_\th$ satisfies (\ref{RN}) in Theorem \ref{FGPSmeas}, 
\begin{equation*}\label{eqmth1}
\frac{dg_\ast \m_\th}{d\m_\th}(\x) \asymp_{\th,R} G(1, g)^{-\th} e^{\b(\th)|g|},
\end{equation*}
for all $\x \in S(g, R)$.
This estimate yields
\begin{equation}\label{eqmth2}
g_\ast \m_\th\left(S(g, R)\right) \asymp_R G(1, g)^{-\th} e^{\b(\th)|g|}\m_\th\left(S(g, R)\right).
\end{equation}
The rest is proceeded as in the proof of \cite[Proposition 6.1]{C}.
Let $0 \le m_0<1$ be the largest $\m_\th$ mass of the point in $\partial \G$, and fix $m$ such that $m_0<m<1$.
By the compactness of $\partial \G$, there exists $r_0>0$ small enough such that every ball of radius $r_0$ has the measure $\m_\th$ less than $m$.
Then there exists $R_0$ depending on $r_0$ such that for every $R \ge R_0$ and for every $g$, the diameter of $\partial \G \setminus g^{-1}S(g, R)$ is less than $r_0$ by \cite[Lemme 6.3]{C}.
Since $0<1-m \le \m_\th\left(g^{-1}S(g, R)\right)$, by (\ref{eqmth2}), it holds that
$\m_\th(S(g, R)) \asymp_R G(1, g)^{\th} e^{-\b(\th)|g|}$.
Thus the corollary follows.
\end{proof}

The following is suggested by Gou\"ezel \cite{G2}.
In particular, probability measures satisfying (\ref{RN}) form a unique measure class.

\begin{theorem}\label{uniqueness}
A probability measure satisfying (\ref{RN}) in Theorem \ref{FGPSmeas} for some $\th$ is doubling.
Arbitrary two probability measures satisfying (\ref{RN}) for the same $\th$ are mutually absolutely continuous and their densities are uniformly bounded from above and from below.
\end{theorem}

\begin{proof}
Let $\m_\th$ be a probability measure satisfying (\ref{RN}).
We note that $\m_\th$ is doubling on shadows in the sense that:
$
\m_\th\left(S(g, 2R)\right) \asymp_{\th ,R} \m_\th\left(S(g, R)\right)
$
for all $g$ in $\G$ and for a fixed $R\ge R_0$ by Corollary \ref{Gibbsth}.
By using this estimate iteratively (if necessary) and by Lemma \ref{shadow}, one shows that $\m_\th$ is doubling, i.e., 
there exists a constant $C=C_{\th, R}$ such that
$
\m_\th\left(B(\x, 2r)\right) \le C\m_\th\left(B(\x, r)\right)
$
for all $\x$ in $\partial \G$ and for all $r\ge 0$.

In the same way, since $\m_\th\left(S(g, R)\right) \asymp_\th \m'_\th\left(S(g, R)\right)$ for arbitrary two probability measures $\m_\th$ and $\m'_\th$ satisfying (\ref{RN}),
two such measures are comparable on balls, i.e., $\m_\th\left(B(\x, r)\right) \asymp_\th \m'_\th\left(B(\x, r)\right)$.
For every Borel set $A$ in $\partial \G$, and for all small enough $r>0$, 
take a compact set $K$ and an open set $U$ such that
$K \subset A \subset U$, 
and
$\m_\th\left(U\setminus K\right)$ and $\m'_\th\left(U\setminus K\right)$ are less than $r$.
Since $K$ is compact, there exists $r_0>0$ such that all balls of radius less than $r_0$ centred in $K$ are included in $U$.
By the doubling property of $\m_\th$, the Vitali covering theorem (e.g.\ \cite[Theorem 1.6]{H}) implies that
there exist countable disjoint balls $\{B_i\}$ of radius less than $r_0$ centred in $K$ such that 
$
\m_\th\left(K \setminus \bigcup_i B_i \right)=0.
$
Therefore
$$
\m_\th(K) \le \sum_i \m_\th(B_i) \le C\sum_i \m'_\th(B_i) = C\m'_\th\left(\cup_i B_i\right) \le C\m'_\th(U).
$$
By letting $r$ tend $0$, one obtains 
$
\m_\th(A) \le C\m'_\th(A).
$
The reverse inequality is also obtained in the same way; 
we conclude the proof.
\end{proof}

Let us recall that a measure on the boundary $\partial \G$ is called ergodic for the action of $\G$, if every $\G$-invariant Borel set has the measure null or co-null.

\begin{corollary}\label{ergodic}
For every $\th \in \R$,
every probability measure $\m_\th$ satisfying (\ref{RN}) in Theorem \ref{FGPSmeas} is ergodic for the action of $\G$.
\end{corollary}

\begin{proof}
Let $A$ be a Borel set which is $\G$-invariant in $\partial \G$, i.e., 
$gA=A$ for every $g$ in $\G$.
Suppose that $\m_\th(A)>0$.
Since the restriction $\m_\th|_A$ of $\m_\th$ to the set $A$ satisfies (\ref{RN}) for the same $\th$,
we have $\m_\th|_A \asymp \m_\th$, by Theorem \ref{uniqueness};
this implies that $\m_\th(\partial\G\setminus A)=0$. 
\end{proof}

The Gibbs property of $\m_\th$ provides the following strong estimate for $\sum_{x \in S_n}G(1, x)^\th$.

\begin{proposition}\label{comp}
For every $\th \in \R$, 
$$
\sum_{x \in S_n}G(1, x)^\th \asymp_\th e^{n\b(\th)}.
$$
\end{proposition}

\begin{proof}
The proof is similar to the one of Lemma \ref{th01} for $\b(1)=0$.
The shadows $S(x, R)$ $(x \in S_n \text{ and } R > \max\{4\d, R_0\})$ cover $\partial \G$ with bounded overlaps $N$.
From Corollary \ref{Gibbsth},
we have
\begin{align*}
1=\m_\th(\partial \G)	\le \sum_{x \in S_n}\m_\th\left(S(x, R)\right)
					\le C\sum_{x \in S_n}G(1,x)^\th e^{-n\b(\th)},
\end{align*}
and,
\begin{equation*}
\sum_{x \in S_n}G(1, x)^\th \le C\sum_{x \in S_n}e^{\b(\th)n}\m_\th\left(S(x, R)\right) \le CNe^{n\b(\th)}.
\end{equation*}
Thus we obtain the claim.
\end{proof}

\section{Symbolic dynamics}\label{Sym}

\subsection{Automatic structure}

Suppose that $S$ is a finite symmetric set of generators in $\G$.
An automaton is a directed graph $\A=(V, E, s_\ast)$
with a distinguished vertex $s_\ast$ called the initial state,
and a labeling on edges by generators $\a: E \to S$.
For a path $\o$ in the graph, a sequence of edges $e_0, e_1, \dots, e_{n-1}$, where the terminus of $e_i$ is the origin of $e_{i+1}$,
we associate a path $\a(\o)$ in the Cayley graph of $\G$
starting from the identity: $1, \a(e_0), \a(e_0)\a(e_1), \dots, \a(e_0)\a(e_1)\cdots \a(e_{n-1})$.
Denote by $\a_\ast(\o)$ the terminus of $\a(\o)$.

\begin{definition}
We say that a finitely generated group $\G$ has a {\it strongly Markov automatic structure} 
if there exists an automaton $\A=(V, E, s_\ast)$ such that
\begin{itemize}
\item[(1)] Every $v$ in $V$ can be reached by a path $\o$ starting from the initial state $s_\ast$.
\item[(2)] For every directed path $\o$ in $\A$, the path $\a(\o)$ is a geodesic in $\G$.
\item[(3)] The map $\a_\ast$ evaluating the terminus of the path is one-to-one and onto from the set of paths starting from $s_\ast$ to $\G$.
\end{itemize}
\end{definition}
By definition, the map $\a_\ast$ gives a bijection from the set of paths from $s_\ast$ of length $n$ to the sphere $S_n$ (the set of words of length $n$) in $\G$.

A theorem by Cannon ensures that every hyperbolic group has a strongly Markov automatic structure \cite{Can}.
In the sequel, we fix an automaton $\A$ for $\G$.
Let us denote by $\SS^\ast$ the set of finite paths in the graph $\A$ (not necessarily starting at $s_\ast$),
by $\SS$ the set of semi-infinite paths and by $\overline \SS=\SS^\ast \cup \SS$ the union of those sets.
We define the metric $d_{\overline \SS}(\o, \o')=2^{-n}$ in $\overline \SS$,
where $\o$ and $\o'$ coincide just until the $n$-th entry.
To a semi-infinite path in the graph $\A$,
we can associate a geodesic ray starting from the identity in the group $\G$,
and a point in the boundary as the extreme point.
The map $\a_\ast$ is extended from $\overline \SS$ to $\G\cup\partial \G$ in this way.
We note that the map $\a_\ast$ is continuous, in fact, H\"older.

\begin{lemma}\label{alphaH}
The map $\a_\ast: \overline \SS \to \G\cup\partial \G$ is H\"older continuous.
\end{lemma}
\begin{proof}
This follows from $d_\e(\x, \y) \le Ce^{-\e(\x|\y)}$.
\end{proof}

Our analysis will be built on the structure of the graph $\A$.
A component is a maximal induced subgraph of $\A$ where every vertex is accessible from every other vertex by a finite directed path.
In general, an automaton has several distinct components which we have to deal with.
A surface group admits an automaton with rather simple structure:
one large component and just single vertices.
For general hyperbolic groups, however, it is not the case, and it is necessary to capture how different components interact each other.

\subsection{Thermodynamic formalism}

In this section, we show that $\b(\th)$ is the pressure for the Ruelle-Perron-Frobenius transfer operator for an appropriate subshift of finite type.
This enables us to prove that $\b(\th)$ is analytic except for at most finitely many points.
The subshift of finite type is constructed on an automatic structure of $\G$, where the thermodynamic formalism is established by Gou\"ezel \cite{G} based on the classical case.
It is not necessarily topologically mixing, or not even recurrent, either.
Below we employ the generalised thermodynamic formalism due to Gou\"ezel, following \cite[Section 3.3]{G}.

For a finite directed graph $\A$ and the set of finite or semi-infinite paths $\overline \SS$ in $\A$,
the shift $\s: \overline \SS \to \overline \SS$ is defined by deleting the first edge of a path.
Given every real-valued H\"older continuous function 
$\f: \overline \SS \to \R$ (a potential),
we define the transfer operator $\Lc_\f$ acting on continuous functions by
$$
\Lc_\f f(\o)=\sum_{\s(\o')=\o}e^{\f(\o')}f(\o'),
$$
where for the empty path $\o=\emptyset$,
we understand that the preimages of the shift $\s$ are taken by the nonempty ones.
The transfer operator encodes the Birkhoff sum $S_n\f(\o)$ in the form that
$$
\Lc_\f^n f(\o)=\sum_{\s^n(\o')=\o}e^{S_n\f(\o')}f(\o'), \text{ where } S_n\f(\o):=\f(\o)+\f(\s \o)+\cdots+\f(\s^{n-1}\o).
$$
The most fundamental case is where the graph $\A$ is topologically mixing, i.e.,
every vertex is accessible from every other vertex (then the graph is called recurrent),
and for arbitrary two vertices in the graph,
there is a path connecting those two vertices of length $n$ for all large enough $n$.
In this case, the Ruelle-Perron-Frobenius theorem describes the spectra of $\Lc_\f$ (\cite[Theorem 1.7]{B}, \cite[Theorem 2.2]{PP}, and \cite[Theorem 3.6]{G}).
In the case where the graph $\A$ is just recurrent,
there is a period $p>1$ such that every loop has the length proportional to $p$.
In this case, the set of vertices of $\A$ is decomposed into $p$ distinct subsets $V_j$ where an edge emanating from a vertex in $V_j$ has the terminus in $V_{j+1}$ for every $j \in \Z/p\Z$.
This decomposition is called a cyclic decomposition of $V$, and the restriction of $\s^p$ to each $V_j$ yields a subshift of finite type which is topologically mixing.
Consider the case where the graph $\A$ is not even recurrent.
Then we decompose the graph $\A$ into components.
To each component $\Cc$,
one can associate 
a transfer operator $\Lc_\Cc$ by restricting $\f$ to paths staying in $\Cc$.
Each transfer operator $\Lc_\Cc$ has finitely many eigenvalues of maximal modulus $e^{Pr_\Cc(\f)}$,
for some real number $Pr_\Cc(\f)$ (the pressure), and they are all simple and isolated.
Define by $Pr(\f)=\max_{\Cc}Pr_\Cc(\f)$
the maximum of the pressure over all components.
A component $\Cc$ is called maximal if $Pr_\Cc(\f)=Pr(\f)$.
The asymptotic behaviour of $\Lc_\f^n$ is controlled by maximal components.

\begin{definition}
The potential $\f$ is called {\it semisimple} if there are no directed paths between arbitrary two distinct maximal components.
\end{definition}

In the case where the potential $\f$ is semisimple, the dominating terms of $\Lc_\f^n$ are fairly well decomposed as in the following:
Let denote by {\it H\"ol} the space of H\"older continuous functions with some fixed exponent with the norm $\|\cdot\|$.

\begin{theorem}[\cite{G}, Theorem 3.8]\label{Gouezel}
Suppose that a potential $\f$ is semisimple.
Denote by $\Cc_1, \dots, \Cc_I$ the maximal components, with the corresponding period $p_i$, and take for each $i$ a cyclic decomposition 
$
\Cc_i=\bigsqcup_{j \in \Z/p_i\Z}\Cc_{i,j}.
$
Then there exist H\"older continuous functions $h_{i,j}$ and measures $\l_{i, j}$ with $\int h_{i,j}d\l_{i, j}=1$ such that for every H\"older continuous function $f$,
\begin{equation*}
\left\|\Lc_\f^n f-e^{nPr(\f)}\sum_{i=1}^{I}\sum_{j=0}^{p_i-1}\left(\int fd\l_{i, (j-n \ {\rm mod} \ p_i)}\right)h_{i,j}\right\|\le C\|f\|e^{-n\e_0}e^{nPr(\f)}.
\end{equation*}
The probability measures 
$\m_i=\frac{1}{p_i}\sum_{j=0}^{p_i-1}h_{i,j}\l_{i,j}$ are invariant under $\s$ and ergodic (but not mixing in general).

Denote by $\Cc_{\rightarrow, i, j}$ the set of edges from which $\Cc_{i, j}$ is accessible with a path of length proportional to $p_i$,
and by $\Cc_{i, j, \rightarrow}$ the set of edges which can be accessible from $\Cc_{i, j}$ by a path of length proportional to $p_i$.
The function $h_{i, j}$ is bounded from below on paths starting with edges in $\Cc_{i, j, \rightarrow}$ 
and the empty path,
and vanishes elsewhere.
The support of the measure $\l_{i, j}$ is the set of infinite paths starting with edges in $\Cc_{\rightarrow, i, j}$
and eventually staying in $\Cc_i$.
\end{theorem}

To each component $\Cc$ (not necessarily maximal), we associate the space of infinite paths staying in $\Cc$, which we denote by $\SS_\Cc$.
There is the Gibbs measure $\m_\Cc$ on $\SS_\Cc$, i.e.,
\begin{equation}
\m_\Cc \left([\o_0, \dots, \o_{n-1}]\right) \asymp e^{-nPr_\Cc(\f)+S_n\f(\o)}
\end{equation}
for every $\o$ in $\SS_\Cc$, 
where $\m_\Cc$ is $\s$-invariant, having the form $(1/p_\Cc)\sum_{j=0}^{p_\Cc-1}h_{\Cc, j}\l_{\Cc, j}$
by abusing the notation in Theorem \ref{Gouezel}.
We formulate a variational principle on each $\SS_\Cc$.
For a finite Borel measurable partition $\Dc$ of $\SS_\Cc$ and a measure $m$ on it, one defines the entropy 
$
H_m(\Dc)=-\sum_{D \in \Dc}m(D)\log m(D).
$
Define the Borel measurable partition of $\SS_\Cc$ by $\Uc:=\{U_x\}_{x \in \Cc}$, where $U_x:=\{\o \in \SS_\Cc \ | \ \o_0=x\}$.
For a $\s$-invariant probability measure $m$ on $\SS_\Cc$, the entropy of $m$ is defined by
$$
h_m(\s)=\lim_{n \to \infty}\frac{1}{n}H_m\left(\Uc\vee \s^{-1}\Uc\vee \cdots \vee \s^{-n+1}\Uc\right),
$$
where $\Dc_1\vee \Dc_2$ denotes the partition refined by the partitions $\Dc_1$ and $\Dc_2$, and the limit exists for every $\s$-invariant probability measure $m$ \cite[Lemma 1.19]{B}.
The following corresponds to \cite[Theorem 1.22]{B}, where the claim is stated for a topologically mixing subshift of finite type, but a recurrent case follows from the same line of the proof.

\begin{proposition}[Variational principle]\label{vp}
For a H\"older continuous potential $\f$ (not necessarily semisimple), and for each component $\Cc$, one has
\begin{equation*}
\max\left\{h_m(\s)+\int_{\SS_\Cc}\f dm \right\}=Pr_\Cc(\f),
\end{equation*}
where the maximum is taken over all $\s$-invariant probability measures $m$ on $\SS_\Cc$ and attained by the Gibbs measure $\m_\Cc$.
\end{proposition}

In particular, $Pr(\f)$ is attained by the Gibbs measure corresponding to a maximal component.
The following proposition describes how the transfer operator behaves under perturbations of the potential.

\begin{proposition}[\cite{G}, Proposition 3.10]\label{pert}
Assume that a potential $\f$ in {\it H\"ol} is semisimple, and $\Cc_1, \dots, \Cc_I$ are the corresponding maximal components as in Theorem \ref{Gouezel}.
There exist $\e_0>0$ and $C>0$ such that
for all $\p$ small enough in {\it H\"ol},
there exist H\"older continuous functions $h_{i, j}^\p$ and measures $\l_{i, j}^\p$ with the same support as $h_{i, j}$ and $\l_{i, j}$, respectively, and real numbers $Pr_{\Cc_i}(\f+\p)$ satisfying that for every $f$ in {\it H\"ol},
\begin{equation*}
\left\|\Lc_{\f+\p}^n f - \sum_{i=1}^{I}e^{nPr_{\Cc_i}(\f+\p)}\sum_{j=0}^{p_i-1}\left(\int fd\l^\p_{i, (j-n \ {\rm mod} \ p_i)}\right)h^\p_{i,j}\right\|\le C\|f\|e^{-n\e_0}e^{nPr(\f)}.
\end{equation*}
The maps $\p \mapsto Pr_{\Cc_i}(\f+\p)$, $\p \mapsto h^\p_{i, j}$ and $\p \mapsto \l^\p_{i, j}$ are analytic in the norm sense from a small ball centred at $0$ in {\it H\"ol} to $\R$, {\it H\"ol} and the dual of {\it H\"ol}, respectively.
Moreover, 
\begin{equation*}
Pr_{\Cc_i}(\f+\p)=Pr(\f)+\int\p d\m_i+O\left(\|\p\|^2\right),
\end{equation*}
where $\m_i=\frac{1}{p_i}\sum_{j=0}^{p_i-1}h_{i, j}\l_{i, j}$.
\end{proposition}

The proof is based on the decomposition on each component.
Note that the pressures $Pr_{\Cc_i}(\f+\p)$ are possibly different.
One can also show that $\f+\p$ is semisimple for all small enough $\p$,
since the pressures of $\f+\p$ on the components other than the maximal ones are bounded away from $Pr(\f)$.
This implies that the maximal components of $\f+\p$ appear within those of $\f$ for all small enough $\p$.

We will need the following lemma to determine whether the potential is semisimple or not.
Let $1_{[E_\ast]}$ be the indicator function on $\overline \SS$ defined on the set of paths starting at the initial state $s_\ast$.
By the definition of the transfer operator, we have
$
\Lc_\f^n1_{[E_\ast]}(\emptyset)=\sum e^{S_n \f(\o)},
$
where the summation is taken over all paths starting at $s_\ast$ of length $n$.

\begin{lemma}\label{semisimple}
Let $Pr(\f)=\max_\Cc Pr_\Cc(\f)$. 
Suppose that there are $L$ components in a finite directed graph $\A$.
We have
$$
\Lc_\f^n1_{[E_\ast]}(\emptyset) \le Cn^Le^{nPr(\f)}.
$$
On the other hand,
if there is a path from the initial state $s_\ast$ to successively $k$ different maximal components,
then we have
$$
\Lc_\f^n1_{[E_\ast]}(\emptyset) \ge C'n^{k-1}e^{nPr(\f)}.
$$
\end{lemma}

\begin{proof}
This is  \cite[Lemma 3.7]{G}, in particular, the latter is a special case of it.
We give the proof of the former for completeness.
Observe that in the components graph which is the directed graph obtained by identifying each component with a point, there are no loops, and thus, there are only finitely many directed paths.
Decompose a path $\o$ in the automaton:
When a path $\o$ starting from $s_\ast$ goes through $k$ different components $\Cc_1, \dots, \Cc_k$ successively, we decompose a path in the automaton by $\o=(u_0, v_1, u_1, \dots, u_{i-1}, v_i, u_i, \dots, u_{k-1}, v_k, u_k)$,
where each $v_i$ is a path in $\Cc_i$, and $u_0, u_1, \dots, u_{k-1}$ are transient paths connecting components $\Cc_1, \dots, \Cc_k$ successively and $u_k$ is a path emanating from $\Cc_k$, possibly empty.
By the H\"older continuity of $\f$, we get
\begin{align*}
S_n\f(\o)	\le Ck + \sum_{i}S_{|v_i|}\f(v_i, u_i, v_{i+1}, \dots, u_k) 
			\le C'k + \sum_{i}S_{|v_i|}\f(v_i),
\end{align*}
and thus,
\begin{align*}
\Lc_\f^n1_{[E_\ast]}(\emptyset)\le \sum e^{Ck}\sum_{a_1+\cdots+a_k=n-b}\left(\sum_{|v_1|=a_1}e^{S_{a_1}\f(v_1)}\right)\cdots\left(\sum_{|v_k|=a_k}e^{S_{a_k}\f(v_k)}\right),
\end{align*}
where $b=\sum|u_i|$ and the first summation is taken over all possible paths in the components graph.
For each component $\Cc$, the eigenvalue of maximal modulus of $\Lc_\f$ restricted on $\Cc$,
is not greater than $e^{Pr(\f)}$,
hence we obtain
$
\sum_{|v_i|=a_i}e^{S_{a_i}\f(v_i)}\le Ce^{a_iPr(\f)}.
$
Therefore it holds that
\begin{align*}
\Lc_\f^n1_{[E_\ast]}(\emptyset)	\le \sum e^{Ck}\sum_{a_1+\cdots+a_k=n-b} e^{(n-b)Pr(\f)} 
								\le Cn^Le^{nPr(\f)},
\end{align*}
and the lemma follows.
\end{proof}

Let $\A=(V,E,s_\ast)$ be a strongly Markov automaton associated with a non-elementary hyperbolic group $\G$,
and $K(x, y)$ be the Martin kernel corresponding to a finitely supported admissible probability measure $\m$ on $\G$.
We define the pressure by
\begin{equation}\label{potential}
\f(\o)=-\log K(\a_\ast \o_0, \a_\ast \o),
\end{equation}
where a path $\o=(\o_0, \o_1, \dots)$ is a sequence of edges in the graph $\A$.
By the cocycle identity:
$
K(gg', \x)=K(g, \x)K(g', g^{-1}\x),
$
we have
\begin{equation}\label{cocycle}
G(1, \a_\ast \bar \o_n) \asymp e^{S_n\f(\o)}
\end{equation}
for the path $\bar \o_n$ cut out from $\o$ at the length $n$, where we used $|S_n \f(\o)-S_n\f(\bar \o_n)| \le C$ by the H\"older continuity of $\f$.

By Lemma \ref{alphaH}, the map $\a_\ast: \overline \SS \to \G\cup \partial \G$ is H\"older continuous, and thus the pressure (\ref{potential}) is also H\"older continuous function on $\overline\SS$.

\begin{proposition}\label{prss}
Let $\f$ be the potential defined in (\ref{potential}).
For every $\th \in \R$, it holds that 
$
\b(\th)=Pr(\th \f).
$
In particular, $\b$ is analytic except for at most finitely many points.
Moreover, $\th \f$ is semisimple for every $\th \in \R$.
\end{proposition}

\begin{proof}
For each component $\Cc$, the pressure $Pr_\Cc(\th \f)$ is analytic in $\th$; hence $Pr(\th \f)$ is analytic in $\th$ except for at most finitely many points since it is defined as the maximum of finitely many analytic functions.
For a path $\o$ of finite length,
the pressure has the following form:
$$
\f(\o)=-\log \frac{G(\a_\ast \o_0, \a_\ast \o)}{G(1, \a_\ast \o)}.
$$
Therefore we obtain
$$
\frac{1}{G(1,1)^\th}\sum_{x \in S_n}G(1,x)^\th=\Lc_{\th \f}^n1_{[E_\ast]}(\emptyset).
$$
By Lemma \ref{semisimple} and Proposition \ref{comp}, 
we have $\b(\th) = Pr(\th \f)$.
Observe that if there is a path from a maximal component to the other maximal component, then concatenating a path emanating from $s_\ast$, one can assume that the path is starting from $s_\ast$ and going through two distinct maximal components.
If there is a directed path from $s_\ast$ passing through $k>1$ different maximal components, then
by the second inequality of Lemma \ref{semisimple}, one has $e^{n\b(\th)} \ge Cn^{k-1}e^{nPr(\th \f)}$;
a contradiction.
\end{proof}

Thermodynamic formalism gives some additional information about a probability measure $\m_\th$ constructed in Theorem \ref{FGPSmeas}.
Henceforth, we consider the potential $\f$ defined in (\ref{potential}).
Fix $\th \in \R$, let $p_1, \dots, p_I$ be the corresponding periods of maximal components for the pressure $\th \f$,
and denote by $p$ the least common multiple of those periods.
Define the measures $\tilde{m}_{np+q}$, $q=0, \dots, p-1$, by the positive linear functionals on $\overline \SS$:
$\tilde{m}_{np+q}(f)=e^{-(np+q)Pr(\th)}\Lc_{\th \f}^{np+q}(1_{[E_\ast]}f)(\emptyset)$.
For each fixed $q$, the measures $\tilde{m}_{np+q}$ converge to some measure $\tilde{m}_{q}$ according to Theorem \ref{Gouezel}.
Then the Ces\`aro average of the measures $(\sum_{n=1}^{N}\tilde{m}_{n})/(\sum_{n=1}^{N}\tilde{m}_{n}(\overline \SS))$
converge to a probability measure weakly.
(First, those measures converge for each H\"older continuous function $f$, then we deduce the weak convergence by approximating each continuous function $f$ by H\"older continuous ones on $\overline \SS$.)
We denote the weak limit of the measures $(\sum_{n=1}^{N}\tilde{m}_{n}(1_{[E_\ast]} \cdot ))/(\sum_{n=1}^{N}\tilde{m}_{n}(1_{[E_\ast]}))$  by $\tilde m_\th$.

\begin{theorem}\label{am=mth}
The measures $\m_{\th, s}$ defined by (\ref{FGPSs}) in Theorem \ref{FGPSmeas} converge to $\a_\ast \tilde m_\th$ as $s \downarrow \b(\th)$ without passing to a subsequence, where $\a_\ast \tilde m_\th$ denotes the push-forward of $\tilde m_\th$ by the map $\a_\ast$.
In particular, the measure $\a_\ast \tilde m_\th$ is comparable with $\m_\th$.
\end{theorem}

\begin{proof}
In the above construction of the measure $\tilde m_\th$, take $f\circ \a_\ast $ for every continuous function $f$ on $\G \cup \partial \G$.
By the relation
$
\sum_{x \in S_n}G(1, x)^\th f(x)=G(1,1)^\th \Lc_{\th \f}^n(1_{[E_\ast]}f\circ \a_\ast)(\emptyset),
$
we deduce that $\m_{\th,s}$ converges weakly to $\a_\ast \tilde m_\th$ as $s \downarrow \b(\th)$.
As we see in the proof of Theorem \ref{FGPSmeas}, every weak limit of $\m_{\th, s}$ as $s \downarrow \b(\th)$ satisfies (\ref{RN}).
Therefore $\a_\ast \tilde m_\th$ is comparable with $\m_\th$ by Theorem \ref{uniqueness}.
\end{proof}

The ergodicity of $\m_\th$ provides the differentiability of $\b$ at $\th \in \R$.

\begin{proposition}\label{ergodic-differential}
If $\m_\th$ is ergodic for the action of $\G$ for some $\th \in \R$, then 
the corresponding $\b$ is differentiable at $\th$.
\end{proposition}

\begin{proof}
Recall that $\b(\th)=\max_\Cc Pr_\Cc(\th \f)$ by Proposition \ref{prss}.
For each $\th \in \R$, let $\Cc_1, \dots, \Cc_k$ be the corresponding maximal components for the potential $\th \f$.
By Proposition \ref{pert}, for each $\Cc_i$, we have
$
Pr'_{\Cc_i}(\th \f)=\int \f d\m_i
$
(the derivative at $\th$).
We will show that it holds that $\int \f d\m_i=\int \f d\m_{i'}$ for $i \neq i'$, and then the differentiability of $\b$ at $\th$ will follow. 
This is nothing but \cite[Proposition 3.16]{G}.
(The argument is well adapted without almost any modifications.)
Below we describe the proof for convenience, following \cite{G} and the idea of Calegari and Fujiwara \cite[Lemma 4.24]{CF}.
Let $c_i:=\int \f d\m_i$.
We define $E_i$ the set of points $\x$ in $\partial \G$ such that the sequence $\log F(1, \g(n))/d(1, \g(n))$ converges to $c_i$ along some geodesic $\g$ towards $\x$.
Notice that in the above, one can replace some geodesic by every geodesic towards $\x$, and in particular, the set $E_i$ is $\G$-invariant.
We will show that $\m_\th(E_i)=1$. This will imply that $E_i\cap E_{i'} \neq \emptyset$, and thus $c_i=c_{i'}$.
Denote by $O_i \subset \overline \SS$ the set of infinite paths $\o$ such that the Birkhoff averages $(1/n)S_n\f(\o)$ have the limit $c_i$.
The Birkhoff ergodic theorem implies that $\m_i(O_i)=1$.
(Recall that $\m_i$ is $\s$-invariant and ergodic on $\SS$.)
We have $\a_\ast^{-1} E_i=O_i\cap [E_\ast]$, where $[E_\ast]$ is the set of paths starting at $s_\ast$, since $F(1, \a_\ast \bar \o_n) \asymp e^{S_n\f(\o)}$ by (\ref{cocycle}).
The measure $\m_i$ is comparable with the measure $\b_i=\sum_j \l_{i,j}$ on the set of paths staying in the component $\Cc_i$.
This enables one to show that $\b_i(O_i^c)=0$ \cite[p.916]{G}.
We will show that $\m_\th(E_i)>0$. Suppose the contrary.
Since the push-forward of $\b_i(\cdot \cap [E_\ast])$ by $\a_\ast$ is absolutely continuous with respect to $\a_\ast \tilde m_\th$ which is comparable with $\m_\th$ by Theorem \ref{am=mth},
we have $\b_i(O_i\cap[E_\ast])=0$.
The set $O_i^c$ has the $\b_i$-measure $0$ by the above, hence $\b_i([E_\ast])=0$ and this is a contradiction since the measure $\b_i$ has a positive mass on $[E_\ast]$ by Theorem \ref{Gouezel}. 
Above all, we have $\m_\th(E_i)>0$, in fact, the measure equals $1$ if $\m_\th$ is ergodic.
Thus we conclude the proof.
\end{proof}

\begin{corollary}\label{conti-diff}
For every finitely supported admissible probability measure $\m$ on $\G$,
the function $\b(\th)$ is continuously differentiable at every $\th \in \R$.
\end{corollary}
\begin{proof}
Recall that $\b$ is analytic except at most finitely many points by Proposition \ref{prss}.
For every $\th \in \R$, a probability measure $\m_\th$ is ergodic by Corollary \ref{ergodic}; the claim follows from Proposition \ref{ergodic-differential}.
\end{proof}

\section{Hausdorff spectrum}\label{HS}

For every finitely supported admissible probability measure $\m$ on a non-elementary hyperbolic group $\G$ equipped with a word metric, 
let $\n$ be the corresponding harmonic measure on the geometric boundary $\partial \G$.
In this section, we consider the dimensional properties of the harmonic measure $\n$.
First, we discuss about a typical pointwise behaviour of $\n$, and establish a formula on the Hausdorff dimension of $\n$ with the entropy and the drift of the random walk.
Next, we consider an atypical pointwise behaviour of $\n$, and introduce the Hausdorff spectrum (the multifractal spectrum), which describes a variety of pointwise dimensions at $\n$-measure null.
The Hausdorff spectrum is given by the Legendre transform of $\b$, and this allows us to prove that the fundamental inequality becomes the genuine equality if and only if the harmonic measure $\n$ is comparable with the $D$-Hausdorff measure $\H^D$ on the boundary $\partial \G$.
This holds for every finitely supported admissible probability measure $\m$ on $\G$.

\subsection{Hausdorff spectrum of the harmonic measure}

Let $\m$ be an arbitrary finitely supported admissible probability measure on $\G$.
First, we prove the following lemma for $\m$ not necessarily symmetric.
Define $d_G(x, y)=-\log F(x, y)$ for $x, y$ in $\G$.
Here $d_G$ satisfies the triangular inequality in the sense that
$
d_G(x, y) \le d_G(x, z)+d_G(z, y)
$
and the positivity $d_G(x, y)>0$ for $x \neq y$, but does not satisfy the symmetry.

\begin{lemma}\label{qi}
There exist constants $L>0$ and $C \ge 0$ such that for all $x, y$ in $\G$,
$$
L^{-1}d(x, y) - C \le d_G(x, y) \le Ld(x, y).
$$
\end{lemma}

\begin{proof}
Suppose that $d(x, y)=m$.
Let $L:=\max_{s \in S}d_G(1, s)>0$.
By the triangular inequality for $d_G$, we have $d_G(x, y) \le Lm=Ld(x, y)$, 
and thus the second inequality holds.
To show the first inequality, it follows from the fact that the Green function satisfies that $G(x, y) \le Ce^{-\l d(x,y)}$.
See also \cite[Proposition 3.5]{BHM11}.
\end{proof}

Recall that the Hausdorff dimension of the measure $\n$ on the metric space $(\partial \G, d_\e)$ is defined by
\begin{equation*}
\dim_H \n=\inf \{\dim_H C \ | \ \n(C^c) = 0\}.
\end{equation*}

We show that the so-called dimension, entropy and drift formula for the harmonic measure $\n$.

\begin{theorem}\label{dim-ent-speed}
For every finitely supported admissible probability measure $\m$ on every non-elementary hyperbolic group $\G$ equipped with a word metric,
the corresponding harmonic measure $\n$ satisfies that
\begin{equation*}
\lim_{r \to 0} \frac{\log \n\left(B(\x, r)\right)}{\log r}=\frac{h}{\e l},
\end{equation*}
for $\n$-a.e.\ $\x$.
In particular,
$
\dim_H\n=h/(\e l).
$
\end{theorem}

\begin{proof}
Note that by Lemma \ref{shadow}, it is enough to show the claim for the shadows $S(g, R)$: 
Fix an $R > 4\d$ and for $\n$-a.e.\ $\x$ in $\partial \G$, take an arbitrary geodesic $\g$ emanating from $1$ towards $\x$, 
we shall show that 
$
-(1/n)\log \n\left(S(\g(n), R)\right) \sim h/l.
$
Notice that $\n\left(S(\g(n), R)\right) \asymp G(1, \g(n))$ by Theorem \ref{Gibbs}, and thus
$
(1/n)\log \n\left(S(\g(n), R)\right) \sim (1/n)\log G(1, \g(n))
$
for such a geodesic ray.
By Theorem \ref{Kaim}, for $\Pr$-almost every sample path of the random walk, there is a geodesic ray $\g_\o$ from $1$ converging to the point $\x$ in the boundary such that $d(x_n, \g_\o(\lfloor l n \rfloor))=o(n)$, where $l>0$ is the drift.
The triangular inequality of $d_G$ shows that 
$$
|d_G(1, x_n) - d_G(1, \g_\o(\lfloor l n\rfloor))| \le \max\{d_G\left(x_n, \g_\o(\lfloor l n\rfloor)\right), d_G\left(\g_\o(\lfloor l n\rfloor), x_n\right)\}.
$$
The inequality $d_G(x, y) \le Ld(x, y)$ in Lemma \ref{qi} implies that
$
(1/n)d_G(1, \g_\o(\lfloor l n\rfloor)) \sim (1/n)d_G(1, x_n).
$
The drift of the random walk with respect to $d_G$ (the Green speed) coincides with the entropy $h$ \cite[Theorem 1.1]{BHM08}.
Therefore we obtain 
\begin{equation}\label{Greenspeed}
\frac{1}{n}d_G(1, \g_\o(\lfloor l n\rfloor)) \sim h
\end{equation}
 for $\Pr$-a.e.\ $\o$.
Above all, it holds that
$
-(1/\lfloor l n\rfloor)\log \n\left(S(\g_\o(\lfloor l n\rfloor), R)\right) \sim h/l
$
for $\Pr$-a.e.\ $\o$. 
The first assertion yields the second e.g., \cite[Theorem 7.1, Chapter 2, p.42]{Pe}, by adapting to the present setting.
\end{proof}

\begin{remark}
In the above proof, it is essential that the Green speed coincides with the entropy \cite{BHM08}.
We note that it is shown for $\m$ not necessarily symmetric in their paper.
\end{remark}

Let us define Hausdorff spectrum of the measure $\n$.
For every real number $\a$, we consider the following set:
\begin{equation*}
E_\a=\left\{\x \in \partial \G \ \Big| \ \lim_{r \to 0}\frac{\log \n\left(B(\x, r)\right)}{\log r}=\frac{\a}{\e} \right\}.
\end{equation*}
The set $E_\a$ is possibly empty for some $\a$.
We call $\dim_H E_\a$ as the function of $\a$ the Hausdorff spectrum of the measure $\n$.
Let $\a_{\min}:=-\lim_{\th \to +\infty}\b'(\th)$ and $\a_{\max}:=-\lim_{\th \to -\infty}\b'(\th)$.
Denote by $f(\a)=\inf_{\th \in \R}\{\a \th +\b(\th)\}$ the Legendre transform of $\b$ which is concave and continuous on $[\a_{\min}, \a_{\max}]$. 
Equivalently, $f(\a)=-\th \b'(\th)+\b(\th)$ for every $\a=-\b'(\th)$.
Note that the interval $[\a_{\min}, \a_{\max}]$ is a singleton if and only if $\b$ is affine.

\begin{theorem}\label{Haus}
For every finitely supported admissible probability measure $\m$ on every non-elementary hyperbolic group $\G$ equipped with a word metric,
the Hausdorff spectrum of the corresponding harmonic measure $\n$ is given by
\begin{equation*}
\dim_H E_\a=\frac{f(\a)}{\e},
\end{equation*}
for every $\a \in (\a_{\min}, \a_{\max})$.
Moreover, the interval $[\a_{\min}, \a_{\max}]$ is bounded in $\R_{>0}$ and 
$E_\a=\emptyset$ for every $\a \notin [\a_{\min}, \a_{\max}]$.
\end{theorem}

\begin{proof}
Recall that $\b$ is continuously differentiable on $\R$ by Corollary \ref{conti-diff}.
For each $\th \in \R$, the derivative of $\b$ at $\th$ has a strictly negative value since $\b$ is convex and $\b(\th) \to \pm \infty$ (resp.) as $\th \to \mp \infty$ (resp.) by Lemma \ref{th01}.
Let us take $\a=-\b'(\th) \in (\a_{\min}, \a_{\max})$ if $\a_{\min} \neq \a_{\max}$; or, $\a=\a_{\min}$ (or $\a_{\max}$) if $\a_{\min}= \a_{\max}$.
First, we will show that $\m_\th (E_\a)=1$ for such a $\th$.
Consider the following set: 
For an arbitrary small enough $s>0$ and an arbitrary large enough $n$, the set of points $\x$ in $\partial \G$, to which there exists a geodesic ray $\g$ from $1$ converging
such that $\n\left(S(\g(n), R)\right) \ge e^{-n(\a -s)}$.
The measure $\m_\th$ of such a set is bounded from above by for every $r>0$,
\begin{equation}\label{ineqHS}
\m_\th\left(\{\x \in \partial \G \ | \ \n\left(S(\g(n), R)\right) \ge e^{-n(\a -s)}\}\right) \le \int_{\partial \G}e^{n(\a-s)r}\n\left(S(\g(n), R)\right)^r d\m_\th(\x).
\end{equation}
(The map $\x \mapsto \g(n)$ has to be measurable; this is obtained as in Theorem \ref{Kaim}.)
The right hand side of (\ref{ineqHS}) is bounded from above by
$$
\sum_{x \in S_n}e^{n(\a-s)r}\n\left(S(x, R)\right)^r \m_\th \left(S(x, R)\right).
$$
By the Gibbs property of $\m_\th$ (Corollary \ref{Gibbsth}),
we have $\m_\th \left(S(x, R)\right) \asymp G(1, x)^\th e^{-\b(\th)|x|}$.
Together with $\n\left(S(x, R)\right) \asymp G(1, x)$ by Theorem \ref{Gibbs},
the ingredient of the above summation is bounded by $Ce^{n(\a-s)r}G(1, x)^{\th+r}e^{-\b(\th)n}$.
Since $\sum_{x \in S_n}G(1,x)^{\th+r} \asymp e^{\b(\th+r)n}$ by Proposition \ref{comp}, we eventually obtain a bound by
$Ce^{n(\a-s)r}e^{(\b(\th+r)-\b(\th))n}$.
Since $\b$ is differentiable at $\th$, we have 
$
\b(\th+r)-\b(\th)=-\a r+o(r)
$
as $r \downarrow 0$.
Therefore for all small enough $r>0$ there exists a constant $c(s, r)>0$ such that
\begin{equation*}
\m_\th\left(\{\x \in \partial \G \ | \ \n\left(S(\g(n), R)\right) \ge e^{-n(\a -s)}\}\right) \le Ce^{-nc(s,r)}.
\end{equation*}
It follows that
$$
\liminf_{n \to \infty}\frac{\log \n\left(S(\g(n), R)\right)}{-n} \ge \a-s, \ \text{for $\m_\th$-a.e.\ $\x$}.
$$
In a similar way, by taking the reverse inequality $\n\left(S(\g(n), R)\right) \le e^{-n(\a +s)}$ in the beginning and replacing $r$ by $-r$, we also get
$$
\limsup_{n \to \infty}\frac{\log \n\left(S(\g(n), R)\right)}{-n} \le \a+s, \ \text{for $\m_\th$-a.e.\ $\x$}.
$$
We conclude that 
$
\lim_{n \to \infty}(-1/n)\log \n\left(S(\g(n), R)\right)=\a,
$
for $\m_\th$-a.e.\ $\x$.
The local dimension (times the parameter $\e$) of $\m_\th$ is obtained by:
$$
\lim_{n \to \infty}\frac{\log \m_\th \left(S(\g(n), R)\right)}{-n}=\lim_{n \to \infty}\frac{\th \log \n \left(S(\g(n), R)\right)-\b(\th)n}{-n}=\th \a+\b(\th),
$$
for $\m_\th$-a.e.\ $\x$.
Therefore, comparing the shadows with the balls by Lemma \ref{shadow}, we have $\e \dim_H E_\a=\th \a+\b(\th)$. 
(On the estimate of Hausdorff dimension, see, e.g. \cite[Proposition 4.9, p.74]{F}, where the proof is written in the $\R^n$ case, but is also adapted to compact metric spaces.)

Next, we shall show that $[\a_{\min}, \a_{\max}]$ is a bounded interval in $\R_{>0}$.
For each $\th \in \R$, denote by $m_\th$ the Gibbs measure supported on $\SS_\Cc$ corresponding to a maximal component $\Cc$ for $\th \f$.
Let $m_\infty$ be a weak limit of $\{m_\th\}$ as $\th$ tends to $+\infty$.
Since for each $\th$, we have
$
\int \f d m_\th=\b'(\th)
$
by Proposition \ref{pert},
we obtain
$
\int \f d m_{+\infty}=-\a_{\min}.
$
In the same way, we obtain
$
\int \f d m_{-\infty}=-\a_{\max},
$
where $m_{-\infty}$ is a weak limit of $\{m_\th\}$ as $\th$ tends to $-\infty$.
By the variational principle Proposition \ref{vp}, for arbitrary $\th_0$ and $\th$, we have
$$
h_{m_\th}+\th_0 \int \f d m_\th \le \b(\th_0).
$$
Choose a positive $\th_0$ such that $\b(\th_0)<0$.
Since $h_{m_\th}\ge 0$, we have
$
-\int\f dm_\th \ge -\b(\th_0)/\th_0>0
$
for every $\th$, and thus we get
$
\a_{\min} >0.
$
On the other hand, 
$$
\a_{\max}= -\lim_{\th_j \to -\infty}\int\f d m_{\th_j} \le \sup_{\o \in \SS}(-\f(\o)) < \infty.
$$

Finally, we shall show that $E_\a=\emptyset$ for every $\a \notin [\a_{\min}, \a_{\max}]$.
Suppose that for an $\a \notin [\a_{\min}, \a_{\max}]$,
there exists a $\x$ in $\partial \G$ such that
$
\lim_{r \to 0}\left(\log \n\left(B(\x, r)\right)\right)/(\log r)=\a/\e.
$
Then we have an infinite path $\o$ starting at $s_\ast$ such that $\a_\ast \o=\x$ and 
$
(1/n)S_n\f(\o) \to -\a
$
as $n$ tends to $\infty$.
Notice that such an infinite path $\o$ eventually stays in some component $\Cc$ (not necessarily maximal for $\f$).
Define a measure on $\SS$ by
$$
m_{\o, n}:=\frac{1}{n}\sum_{k=0}^{n-1}\d_{\s^k \o}.
$$
Let $m_\o$ be a weak limit of $\{m_{\o, n}\}$.
Note that the measure $m_\o$ is $\s$-invariant and supported in $\SS_\Cc$.
The variational principle Proposition \ref{vp} implies that for every $\th$,
$$
h_{m_\o}(\s)+\th \int \f d m_\o \le \b(\th),
$$
i.e., the line defined by the left hand side lies under the convex curve defined by the right hand side in the above inequality.
Thus there exists a $\th \in [-\infty, +\infty]$ such that
\begin{equation}\label{range}
\int \f dm_\o=\int \f dm_\th,
\end{equation}
where we have $\th=-\infty$ or $+\infty$ when the line is asymptotic to $\b(\th)$.
Here the right hand side of (\ref{range}) equals to $\b'(\th)$ which is in $[-\a_{\max}, -\a_{\min}]$.
Notice that since
$
\int \f d m_{\o, n_j}=(1/n_j)S_{n_j}\f(\o) \to \int \f d m_\o,
$
we have $-\a=\int \f d m_\o$.
Hence it has to hold that $-\a \in [-\a_{\max}, -\a_{\min}]$; a contradiction.
We complete the proof.
\end{proof}

\begin{remark}
In fact, we are able to show that the set of irregular points
$
E_{irr}=\left\{\x \in \partial \G \ | \ \lim_{r \to 0}\left(\log \n\left(B(\x, r)\right)\right)/(\log r) \text{ does not exist.}\right\}
$
is not empty.
Indeed, for each component $\Cc$ (whose cardinality is larger than one), one can find a path $\o$ in $\SS_\Cc$ such that the Birkhoff averages $(1/n)S_n\f(\o)$ do not converge by a result of Barreira and Schmeling \cite[Theorem 2.1]{BS}.
It has been also shown that such paths form a set of full topological entropy [ibid].
It is expected to hold that $\dim_H E_{\a}=f(a)/\e$ for $\a=\a_{\min}$ and $\a_{\max}$ as well, but those would require some additional argument (e.g., \cite[Section 9.2]{PU}).
\end{remark}

\begin{remark}
If an automaton has a unique component whose cardinality is larger than one, then $\b$ is analytic by definition and thus the Legendre transform $f(\a)$ is also analytic on $(\a_{\min}, \a_{\max})$.
This is the case for free groups and for surface groups, i.e., the fundamental groups of closed Riemann surfaces of genus greater than one, for example.
The condition that an automaton has a single non-trivial component corresponds to Assumption 1 in \cite{HMM}.
In general, we have shown that $\b$ is not only continuously differentiable but also analytic except for at most finitely many points, hence
it would be reasonable to expect that $\b$ is also analytic for every non-elementary hyperbolic group and every finitely supported admissible probability measure $\m$ on it.
\end{remark}

\subsection{Proof of Theorem \ref{thm1}}

First, we prove the following proposition which is a slight extension of \cite[Proposition 5.4]{BHM11}.
Let $\check\m$ be the reflected measure of $\m$, i.e., $\check \m(x)=\m(x^{-1})$.
Notice that if $\m$ is finitely supported admissible on $\G$, then the reflected measure $\check \m$ is also the case.
Denote by $\n$ and $\check \n$ the corresponding harmonic measures on $\partial \G$ for $\m$ and $\check \m$, respectively.

\begin{proposition}\label{bounded}
If the $D$-Hausdorff measure $\H^D$ and the harmonic measure $\n$, and $\H^D$ and the reflected harmonic measure $\check \n$ are mutually absolutely continuous, respectively,  
then their Radon-Nikodym derivatives are uniformly bounded from above and from below, i.e., $\H^D \asymp \n$ and $\H^D \asymp \check \n$.
\end{proposition}

\begin{proof}
Define for $x, y$ in $\G$,
$$
\phi(x, y):=\frac{G(x, y)}{G(x, 1)G(1, y)}.
$$
There is an extension of $\phi$ on $\partial^2\G:=\partial \G\times \partial \G \setminus \text{\rm diagonal}$, which is H\"older continuous on each compact subset. 
(See \cite[Proposition 3.17]{L}. It is written for free groups case, but the argument applies to the present setting as well.)
Indeed, take $\x$ and $\y$ in $\partial \G$ ($\x\neq \y$) and $r>0$ such that $B(\x, r) \cap B(\y, r) =\emptyset$ (the balls are defined in $\G\cup \partial \G$).
For every $x \in B(\x, r)\cap \G$ and every $y \in B(\y, r)\cap \G$, the function $\phi(x, y)$ is bounded by the Ancona inequality (\ref{A}) and the Harnack inequality (\ref{H}).
It holds that uniformly for a fixed $x$, the function $y \mapsto G(x, y)/G(1, y)$ is H\"older continuous on $B(\y, r)$, and
uniformly for a fixed $y$, the function $x \mapsto G(x, y)/G(x, 1)=\check G(y, x)/\check G(1, x)$ is also H\"older continuous on $B(\x, r)$ by Theorem \ref{Holder}.
Therefore the function $\phi$ is extended on $\partial^2\G$ and is H\"older continuous on each compact subset.
The measure $d\tilde \n(\x, \y):=\phi(\x, \y) d(\check \n \otimes \n)(\x, \y)$ on $\partial^2\G$ is $\s$-finite and invariant under the diagonal action of $\G$, since
\begin{align*}
\frac{\phi(\x, \y)}{\phi(g^{-1}\x, g^{-1}\y)}		=\lim_{x \to \x, \ y \to \y}\frac{G(x, g)G(g, y)}{G(x, 1)G(1, y)} 
											=\check K(g, \x)K(g, \y) 
											=\frac{dg_{\ast}(\check \n \otimes \n)}{d(\check \n \otimes \n)}(\x, \y).
\end{align*}
Let $\m_0$ be a Patterson-Sullivan measure on $\partial \G$.
Define the measure on $\partial^2\G$ by
$$
d\tilde \m_0 (\x, \y):=\frac{d(\m_0 \otimes \m_0)(\x, \y)}{\exp(-2v (\x|\y))}.
$$
Here the measure $\tilde \m_0$ is $\s$-finite and $\G$-quasi-invariant on $\partial^2\G$, i.e., there exists a constant $C\ge 1$ such that
$
C^{-1}\tilde \m_0(A) \le \tilde \m_0(g A) \le C\tilde \m_0(A)
$
for all $g$ in $\G$ and all Borel sets $A$ in $\partial^2\G$ \cite[Corollaire 9.4]{C}.
Suppose that $\m_0$ and $\n$, and $\m_0$ and $\check \n$ are mutually absolutely continuous, respectively.
We have $d\m_0=fd\n$ and $d\m_0=\check f d\check \n$ for some positive Borel measurable functions $f$ and $\check f$ on $\partial \G$.
We obtain 
$
d\tilde \m_0(\x, \y)=\tilde F(\x, \y)d\tilde \n(\x, \y),
$
where 
$$
\tilde F(\x, \y):=\check f(\x)f(\y)\frac{\exp(2v(\x|\y))}{\phi(\x, \y)}.
$$
First, we shall prove that $\tilde F$ is uniformly bounded from above and from below $\tilde \n$-a.e.
Define the set $A_C:=\{(\x, \y) \in \partial^2 \G \ | \ C^{-1} \le \tilde F(\x, \y) \le C\}$, then there exists a constant $C \ge 1$ such that
$
\tilde \n(A_C)>0.
$
The measure $\tilde \n$ is ergodic for the diagonal action of $\G$ since $\check \n \otimes \n$ is ergodic \cite[Theorem 6.3]{K}.
Therefore, for $\tilde \n$-a.e.\ $(\x, \y)$ in $\partial^2\G$, there exists $g$ in $\G$ such that $g(\x, \y) \in A_C$. 
Since $\tilde \m_0$ is $\G$-quasi-invariant and $\tilde \n$ is $\G$-invariant, we have $\tilde F(g\x, g\y) \asymp \tilde F(\x, \y)$ for all $g$ in $\G$ and for $\tilde \n$-a.e.\ $(\x, \y)$.
Hence there exists a constant $C'\ge 1$ such that $C'^{-1}\le \tilde F(\x, \y) \le C'$ for $\tilde \n$-a.e.\ $(\x, \y)$.
Thus we have
\begin{equation*}
\check f(\x) f(\y) \asymp \frac{\phi(\x, \y)}{\exp(2v(\x|\y))}, \ \ \text{$\tilde \n$-a.e.}
\end{equation*}
Suppose that $f$ is unbounded on $B(\y,r)$ for all small enough $r>0$. Let $\x\neq \y$ such that $\check f(\x)>0$.
The right hand side $\phi(\x, \y')/\exp(2v(\x|\y'))$ is uniformly bounded for $\y' \in B(\y,r)$ for all small enough $r$, while the left hand side $\check f(\x)f(\y')$ is not; a contradiction.
This implies that $f$ (and also $\check f$) has to be bounded on $\partial \G$, and the same argument yields the boundedness of $1/f$ (and $1/\check f$).
The claim follows.
\end{proof}

The logarithmic volume growth $v$ of $(\G, d)$ has an interpretation as the Hausdorff dimension of the boundary $\partial \G$, i.e., 
$
\dim_H \partial \G=v/\e
$
(\cite[Corollaire 7.6]{C} and \cite[Corollary 2.5.10]{Ca}).
This leads a natural implication of the fundamental inequality $h \le l v$ as the inequality of dimensions:
$
\dim_H \n \le \dim_H \partial \G,
$
as it is pointed out by Vershik \cite[Introduction, p.670]{V}.
Below we show that $\n$ is of maximal Hausdorff dimension if and only if $\n$ and $\H^D$ are comparable.
Recall that $D=v/\e$.
The following is a restatement of Theorem \ref{thm1} from Section \ref{intro}.

\begin{theorem}\label{fund}
For every finitely supported admissible probability measure $\m$ on every non-elementary hyperbolic group $\G$ equipped with a word metric,
it holds that
$
h = l v,
$
if and only if the corresponding harmonic measure $\n$ and the $D$-Hausdorff measure $\H^D$ on the boundary $\partial \G$
are mutually absolutely continuous and their densities are uniformly bounded from above and from below, i.e.,
$\n \asymp \H^D$.
\end{theorem}

\begin{remark}
The analogy between the logarithmic volume growth $v$ and the topological entropy is pointed out in \cite[p.681]{V}.
\end{remark}

\begin{proof}[Proof of Theorem \ref{fund}]
Since the inequality 
$
h \le l v
$
is equivalent to
the inequality
$
\dim_H \n \le \dim_H \partial \G
$
by Theorem \ref{dim-ent-speed},
we have the equality 
$
h=l v,
$
as soon as the harmonic measure $\n$ is comparable with the $D$-Hausdorff measure $\H^D$.
We shall prove the converse.
Recall that $\b$ is differentiable at every $\th \in \R$ (Corollary \ref{conti-diff}).
In view of Theorem \ref{Haus},
the Hausdorff spectrum $\dim_H E_\a$ has the maximum $v/\e$ at $\a=-\b'(0)$. 
On the other hand, 
$\dim_H E_\a=h/(\e l)$ at $\a=-\b'(1)$.
Indeed, since $\e \dim_H E_\a=\a$ at $\a=-\b'(1)$ and
$\m_1=\n$, one has $\a=h/l$ (Corollary \ref{FGPSth01} and Theorem \ref{dim-ent-speed}). 
The fundamental observation is that the inequality $h/l \le v$ is strict if and only if $\b$ is strictly convex on $[0,1]$.
Suppose that the equality $h=l v$ holds.
Then $\b$ has to be affine on $[0,1]$ (in fact, on $\R$).
Let us recall that $\b(\th)=Pr(\th \f)=\max_\Cc Pr_{\Cc}(\th \f)$ (Proposition \ref{prss}).
We adapt the notation in Theorem \ref{Gouezel}.
Denote by $\Cc_1, \dots, \Cc_I$ the maximal components at $\th=0$.
Every $Pr_{\Cc_i}(\th \f)$ coincides with $-v \th +v$ on $\R$, 
since each $Pr_{\Cc_i}(\th \f)$ is convex and analytic.
Hence $\Cc_i$, $i=1, \dots, I$ are also precisely maximal components for all $\th \in [0,1]$, in particular, for $\th=1$.
Since we now have $Pr''_{\Cc_i}(\th \f)=0$ for all $\th \in \R$, 
the potential $\f$ is cohomologous to a constant, i.e.,
$\f=a+u\circ \s-u$ for some constant $a$ and some H\"older continuous function $u$ on some subset $\SS'$ of $\SS$,
by \cite[Proposition 4.12]{PP}.
(Although in \cite{PP}, the statement is written for a topologically mixing subshift of finite type, it is adapted to our setting by decomposing into components.)
Here the domain $\SS'$ of $u$ is the union of the sets $\SS_i$ of infinite paths staying in the component $\Cc_i$, i.e., $\SS'=\bigcup \SS_i$.
We use the relation 
$
\Lc_\f^n1_{[E_\ast]}(\emptyset)=\sum e^{S_n\f(\o)},
$
where the summation runs over the paths $\o$ of length $n$ starting at $s_\ast$.
We define the set of paths $\o$ starting at $s_\ast$ satisfying that
for a fixed $K$, the path $\o$ is in one of those components $\Cc_i$ for all $n \ge K$, i.e.,
$$
\SS_K:=\{\o \in \SS \ | \ \text{$\o$ starts at $s_\ast$ and $\o_n \in \bigcup_{i=1}^{I}\Cc_i$ for all $n \ge K$}\}.
$$
Let $\partial \G_K:=\a_\ast (\SS_K)$.
Recall that the components $\Cc_i$, $i=1, \dots, I$, are maximal both at $\th =0$ and at $\th=1$, and $\m_\th \asymp \a_\ast \tilde m_\th$, where $\tilde m_\th$ for $\th=0$ and $1$ are supported in the common set of paths eventually staying in $\bigcup_{i=1}^{I}\Cc_i$.
Therefore we deduce that
\begin{equation}\label{ae}
\partial \G = \bigcup_K \partial \G_K,
\text{ $\H^D$-a.e.\ and also $\n$-a.e.},
\end{equation}
since $\H^D \asymp \m_0$ and $\n=\m_1$ (Corollary \ref{FGPSth01}).
For every $\o \in \SS_K$, we have
\begin{equation}\label{a}
|S_n \f(\o)-an| \le C_K,
\end{equation}
where $C_K$ is a constant depending on $K$.
Indeed, let $\overline \o$ be any extension of $\o$ staying in some component $\Cc_i$, 
it follows that
$
|S_n \f(\o)-S_n \f(\overline \o)| \le C
$
by the H\"older continuity of $\f$, 
and 
$
|S_n \f(\overline \o)-an| \le C_K,
$
by using
$
|S_n \f(\overline \o)-(\f(\s^K\overline \o)+\cdots +\f(\s^{n-1}\overline \o))| \le C_K
$
for all $n \ge K$.
Then we obtain
$$
\Lc_\f^n1_{[E_\ast]}(\emptyset) \ge C_K e^{(a+v)n}.
$$
On the other hand, as in the proof of Lemma \ref{semisimple}, we deduce that there exists some $\e_0>0$ such that
$$
\Lc_\f^n1_{[E_\ast]}(\emptyset) \le Cn^{L}(e^{-n\e_0}+e^{n(a+v)}),
$$
for all $n>0$,
since the summation over paths which do not stay in any $\Cc_i$ contributes as $O(e^{-n\e_0})$ at $\th=1$.
By Proposition \ref{comp}, $\Lc_\f^n1_{[E_\ast]}(\emptyset) \asymp 1$ due to $\b(1)=0$, we conclude that $a=-v$.
Therefore, by (\ref{a}), there exists a constant $C_K>0$ such that 
for every path $\o \in \SS_K$ and
for every $n \ge K$,
we have 
$$
C_K^{-1} e^{-v n} \le e^{S_n\f(\o)} \le C_K e^{-v n}.
$$
Together with
$
G(1, \a_\ast \bar \o_n) \asymp e^{S_n\f(\o)}
$
by (\ref{cocycle}),
we deduce that $G(1, \a_\ast \bar \o_n) \asymp_K e^{-v n}$ for all paths $\o \in \SS_K$ and for every $n \ge K$.
Consider $\partial \G_K$ such that $\n(\partial \G_K)>0$.
For every $\x \in \partial \G_K$,
there exists a path $\o \in \SS_K$ such that $\a_\ast \o=\x$.
Therefore we obtain that there exists a constant $C_K>0$ such that for every $\x \in \partial \G_K$ and for every $n \ge K$,
$$
C_K^{-1} e^{-vn} \le \n\left(S(\a_\ast \bar \o_n, R)\right) \le C_K e^{-vn},
$$ 
by Theorem \ref{Gibbs}, we also have
$
C_K^{-1} r^{v/\e} \le \n\left(B(\x, r)\right) \le C_K r^{v/\e},
$ 
for every $\x \in \partial \G_K$ and for all small enough $r>0$ by Lemma \ref{shadow}.
A standard covering argument yields $\n(\ \cdot \ \cap \partial \G_K) \asymp_K \H^D(\ \cdot \ \cap \partial \G_K)$ (e.g.\ \cite[Corollary 2.5.10]{Ca}).
Since we have (\ref{ae}), 
we conclude that $\H^D$ and $\n$ are mutually absolutely continuous.
Moreover, we shall prove that their Radon-Nikodym derivatives are uniformly bounded from above and from below.
Indeed, notice that the entropy $\check h$ and the drift $\check l$ for the reflected measure $\check \m$ coincide with those for $\m$, respectively, since $\check \m^{\ast n}(x)=\m^{\ast n}(x^{-1})$. 
Thus $h=lv$ is equivalent to $\check h=\check l v$.
As above, we also conclude that $\H^D$ and $\check \n$ are mutually absolutely continuous.
Therefore by Proposition \ref{bounded}, in fact, we have $\H^D \asymp \n$.
\end{proof}

\begin{remark}\label{rigidity}
For a fixed $R \ge R_0$ and for every $g$ in $\G$, we have
$
\H^D(S(g, R)) \asymp e^{-v|g|}
$
by Corollary \ref{Gibbsth} (see also \cite[Proposition 6.1]{C}).
If $\H^D \asymp \n$, then combining with $G(1, g) \asymp \n(S(g, R))$ (Theorem \ref{Gibbs}), we obtain
$
G(1, g) \asymp e^{-v|g|};
$
In other words, there exists a constant $C>0$ such that for every $g$ in $\G$, it holds that
\begin{equation}\label{rigid-estimate}
\left|\log G(1, g)+v|g|\right| \le C.
\end{equation}
The last estimate (\ref{rigid-estimate}) implies that $h=lv$ since the Green speed is the entropy \cite[Theorem 1.1]{BHM08}.
The equivalence between $h=lv$ and (\ref{rigid-estimate}) is proven in \cite[Corollary 1.2 and Theorem 1.5]{BHM11} (see also \cite{Ha}) for every finitely supported admissible symmetric probability measure $\m$ on $\G$.
Compare with results on the rigidity of cocycles in \cite{GMMpre}.
\end{remark}

\section{finitary results}\label{FR}

We show finitary versions of Theorem \ref{dim-ent-speed}, which are inspired by the results on simple random walks on Galton-Watson trees by Lyons, Pemantle and Peres \cite{LPP}.
As a consequence, if $h < l v$, then the random walk is confined in a strictly smaller region in $\G$;
but if $h=l v$, then such confinement does not happen. 
We make it more precise.
Denote by $r$ the diameter of the support of the step distribution $\m$, and by $[S_n]$ the $r$-neighbourhood of the sphere $S_n$ in $\G$ (the annulus).
We define the first hitting time for the random walk $\{x_n\}_{n=0}^{\infty}$ to the annulus $[S_n]$ by
$
\t_n:=\min\{k \ge 0 \ | \ x_k \in [S_n]\},
$
where $\t_n$ is finite $\Pr$-a.s., since the random walk is transient and the annulus is thick enough for the range of the step.
Let $\n_{[n]}$ be the distribution of $x_{\t_n}$, which is supported on $[S_n]$.

\begin{theorem}\label{K}
Let $\hat h:=h/l$. 
\begin{itemize}
\item[(1)] For the sequence of the first hitting points $\{x_{\t_n}\}_{n=0}^{\infty}$,
$$
\lim_{n \to \infty}-\frac{1}{n}\log \n_{[n]}\left(x_{\t_{n}}\right) =\hat h, \ \ \text{$\Pr$-a.s.}
$$
\item[(2)] For every $a>0$, 
$$
\lim_{n \to \infty}\n_{[n]}\left(x \in [S_n] \ \big| \ e^{-n(\hat h+a)} \le \n_{[n]}(x) \le e^{-n(\hat h -a)}\right)=1. 
$$
\item[(3)] For every $a \in (0,1)$,
define $K_n(a)$ as the minimal number of points which form a set carrying the $\n_{[n]}$-measure at least $a$, then
$$
\lim_{n \to \infty}\frac{1}{n}\log K_n(a)=\hat h.
$$
\end{itemize}
\end{theorem}

\begin{proof}
The proof is analogous to \cite[Theorem 9.8]{LPP}.
Note that (1) implies (2) since almost sure convergence yields convergence in probability, and (2) implies (3) directly.
We prove (1).
Recall that the Green speed coincides with the entropy: 
$$
\lim_{n \to \infty}-\frac{1}{n}\log F(1, x_n)=h,
$$
$\Pr$-a.s.\ by \cite[Theorem 1.1]{BHM08}.
Since we have $\n_{[n]}(x_{\t_n}) \le F(1, x_{\t_n})$, and $\t_n/n\to 1/l$ as $n$ tends to infinity $\Pr$-a.s., it holds that
$$
\liminf_{n \to \infty} -\frac{1}{n}\log \n_{[n]}\left(x_{\t_n}\right) \ge \frac{h}{l}=\hat h,
$$
$\Pr$-a.s.
For every $s> \hat h$ and every $0<t<(s -\hat h)/2$, define the set
$$
A_n=\{x \in [S_n] \ | \ \n_{[n]}(x) \le e^{-n(s-t)}, F(1, x) \ge e^{-n(\hat h+t)}\}.
$$
We have 
$
\sum_{x \in [S_n]}F(1, x) \asymp 1
$
by Proposition \ref{comp}, and thus
we deduce that
$
\n_{[n]}\left(A_n\right) \le C e^{-n(s-\hat h -2 t)}.
$
Therefore, by the Borel-Cantelli lemma, $\Pr$-a.s., there exists $N$ such that for all $n \ge N$ the point $x_{\t_n}$ is not in $A_n$.
Since we have $F(1, x_{\t_n}) \ge e^{-n(\hat h+t)}$ eventually $\Pr$-a.s.,
we obtain $\n_{[n]}(x_{\t_n}) \ge e^{-n(s-t)}$ eventually $\Pr$-a.s.
Hence it holds that
$$
\limsup_{n \to \infty}-\frac{1}{n}\log \n_{[n]}\left(x_{\t_n}\right) \le s-t,
$$
$\Pr$-a.s., by taking $s$ arbitrary close to $\hat h$ and $t$ arbitrary close to $0$, the theorem follows.
\end{proof}

We show a ``confinement" result of the random walk, stated in Theorem \ref{thm3}.
The following is known for almost every realisation of Galton-Watson tree conditioning on having infinite trees (\cite[Theorem 9.9]{LPP} and \cite[Theorem 16.30]{LP}).

\begin{theorem}\label{confinement}
For every $a \in (0,1)$, there exists a subset $\G_a$ in $\G$ such that 
$$
\Pr\left(x_n \in \G_a \text{ for all $n \ge0$}\right) \ge 1-a,
$$
and
$$
\lim_{n \to \infty}\frac{1}{n}\log \left|\G_a \cap S_n\right|=\hat h.
$$
\end{theorem}

In particular, if $\hat h=h/l < v$, then the random walk stays in an exponentially small region $\G_a$ of logarithmic volume growth $\hat h$ with positive probability, and moreover, by Theorem \ref{K}, there is no region of logarithmic volume growth strictly less than $\hat h$ where the random walk stays with positive probability.
The confinement of the random walk occurs if and only if the inequality $h < l v$ is strict as it is also mentioned in \cite[p.683]{V}.

\begin{proof}[Proof of Theorem \ref{confinement}]
By Theorem \ref{Kaim}, $\Pr$-a.s., there exists a geodesic ray $\g_\o$ starting at $1$ such that $(1/n)d\left(\g_\o(\lfloor l n \rfloor), x_n\right) \to 0$ as $n$ tends to infinity, where $l>0$ is the drift of the random walk.
Take the map $\o \mapsto \g_\o$ as a Borel measurable map as in Theorem \ref{Kaim}.
As (\ref{Greenspeed}) in the proof of Theorem \ref{dim-ent-speed}, we have
$$
-\frac{1}{\lfloor l n\rfloor}\log G(1, \g_\o(\lfloor l n \rfloor)) \to \hat h,
$$
as $n$ tends to infinity, $\Pr$-a.s.
According to the Egorov theorem, one can make the above two convergences uniform on an event with arbitrary high probability. 
Namely, for every $a \in (0,1)$ there exists an event $A_a \subset \O$ such that $\Pr(A_a)\ge 1-a$ and
for some sequence $\d_n$ decreasing to $0$, for all $\o \in A_a$,
\begin{equation}\label{uniform}
\left|-\frac{1}{\lfloor l n\rfloor}\log G(1, \g_\o(\lfloor l n \rfloor)) - \hat h\right| \le \d_{\lfloor l n\rfloor}
\text{ and }
d\left(\g_\o(\lfloor l n \rfloor), x_n\right)\le \lfloor l n\rfloor \d_{\lfloor l n\rfloor}.
\end{equation}
Let $V_a:=\left\{x \in \G \ | \ G(1, x) \ge e^{-|x|(\hat h+\d_{|x|})}\right\}$. Define 
$
\G_a:=\bigcup_{x \in V_a}B\left(x, |x|\d_{|x|}\right).
$
For every $\o \in A_a$ and for every $n \ge 0$, the point $\g_\o(\lfloor l n \rfloor)$ is in $V_a$,
and thus, for every $\o \in A_a$, it holds that $x_n \in \G_a$ for all $n \ge 0$ 
by (\ref{uniform}).
Hence
$$
\Pr\left(x_n \in \G_a \text{ for all $n \ge0$}\right) \ge 1-a.
$$

We deduce that $|V_a \cap S_n| \le Ce^{n(\hat h+\d_n)}$ since $\sum_{x \in S_n}G(1, x) \asymp 1$ by Proposition \ref{comp}.
Note that for all small enough $t>0$ and for all large enough $n$, we have
$
|B(1, n\d_{n})| \le e^{nt}.
$
Therefore for all large enough $n$,
$$
|\G_a \cap S_n|\le \sum_{\frac{n}{1+t} \le k \le \frac{n}{1-t}}Ce^{k t}e^{k(\hat h+t)} \le Ce^{\frac{n}{1-t}(\hat h+2t)},
$$
and it follows that
$$
\limsup_{n \to \infty}\frac{1}{n}\log |\G_a\cap S_n| \le \frac{\hat h+2t}{1-t}.
$$
The above holds for all small enough $t>0$, then we obtain the upper bound by $\hat h$.
The lower bound follows by $|\G_a \cap [S_n]|\ge K_n(1-a)$ and by Theorem \ref{K} (3).
\end{proof}

\textbf{Acknowledgements.}
The author would like to thank S\'ebastien Gou\"ezel for pointing out some errors in the first draft, numerous valuable suggestions and comments to improve the results in this paper, and also for sending him a preliminary version of the paper \cite{GMMpre}, Vadim A.\ Kaimanovich for his suggestion on a possible connection to multifractal analysis, Pierre Mathieu for helpful comments and discussion about the present results, Art\"em Sapozhnikov for discussion on random walks on hyperbolic groups, Takefumi Kondo for helpful comments, 
J\'er\'emie Brieussel for carefully reading an earlier version of this paper and for helpful comments to improve the presentation, and Yuval Peres for helpful conversation and for introducing to him many fractal people in Jerusalem.


\begin{thebibliography}{99}

\bibitem[Anc]{A} Ancona A., {\it Positive harmonic functions and hyperbolicity}, Potential Theory, Prague 1987, Lecture Notes in Mathematics, vol. 1344, Springer, Berlin, 1988, 1-23.

\bibitem[Ave]{Av} Avez A., {\it Entropie des groupes de type fini}, C. R. Acad. Sci. Paris S\'er. A {\bf 275}, 1363-1366 (1972).

\bibitem[BS]{BS} Barreira L., Schmeling J., {\it Sets of ``non-typical" points have full topological entropy and full Hausdorff dimension}, Israel J. Math., {\bf 116}, 29-70 (2000).

\bibitem[BHM1]{BHM08} Blach\`ere S., Haissinsky P., Mathieu P., {\it Asymptotic entropy and Green speed for random walks on countable groups}, Ann. Prob., Vol. 36, No.3, 1134-1152 (2008).

\bibitem[BHM2]{BHM11} Blach\`ere S., Haissinsky P., Mathieu P., {\it Harmonic measures versus quasiconformal measures for hyperbolic groups}, Ann. Sci. \'Ec. Norm. Sup\'er. (4) {\bf 44}, 683-721 (2011).

\bibitem[Bou]{Bou} Bourgain J., {\it Finitely supported measures on $SL_2(\R)$ which are absolutely continuous at infinity}, Geometric aspect of functional analysis (B. Klartag et al. eds.), 133-141, Lecture Notes in Math. 2050, Springer-Verlag, Berlin Heidelberg, 2012.

\bibitem[Bow]{B} Bowen R., {\it Equilibrium states and the ergodic theory of Anosov diffeomorphisms}, Lecture Notes in Mathematics 470, Second revised edition, Springer-Verlag Berlin Heidelberg, 2008.

\bibitem[BT]{BT} Brieussel J., Tanaka R., {\it Discrete random walks on the group Sol}, arXiv:1306.6180v1 [math.PR] (2013), to appear in Israel J. Math.

\bibitem[Cal]{Ca} Calegari D., {\it The ergodic theory of hyperbolic groups}, 
Geometry and Topology Down Under, Contemp. Math. {\bf 597}, Amer. Math. Soc., Providence, RI, 2013, 15-52.
%arXiv:1111.0029v2 [math.GR] (2012).

\bibitem[CF]{CF} Calegari D., Fujiwara K., {\it Combable functions, quasimorphisms, and the central limit theorem}, Ergodic Theory Dynam. Systems, {\bf 30}, 1343-1369 (2009).

\bibitem[Can]{Can} Cannon J W., {\it The combinatorial structure of cocompact discrete hyperbolic groups}, Geom. Dedicata {\bf 16}, 123-148 (1984).

\bibitem[CM1]{CM1} Connell C., Muchnik R., {\it Harmonicity of quasiconformal measures and Poisson boundaries of hyperbolic spaces}, Geom. Funct. Anal. Vol. 17, 707-769 (2007).

\bibitem[CM2]{CM2} Connell C., Muchnik R., {\it Harmonicity of Gibbs measures}, Duke Math. J. Vol. 137, No. 3, 461-509 (2007).

\bibitem[Coo]{C} Coornaert M., {\it Mesures de Patterson-Sullivan sur le bord d'un espace hyperbolique an sens de Gromov}, Pacific J. Math., Vol 159, No 2, 241-270 (1993).

\bibitem[Der]{D} Derriennic Y., {\it Quelques applications du th\'eor\`eme ergodique sous-additif}, Conference on Random Walks (Kleebach, 1979), 183-201, Ast\'erisque, {\bf 74}, 
Soc. Math. France, Paris, 1980.

\bibitem[Fal]{F} Falconer K., {\it Fractal Geometry, Mathematical Foundations and Applications}, Third Edition, John Wiley \& Sons, Chichester 2014.

\bibitem[Fen]{Fe} Feng D-J., {\it Gibbs properties of self-conformal measures and the multi fractal formalism}, Ergodic Theory Dynam. Systems {\bf 27}, no.3,  787-812 (2007).

\bibitem[GdlH]{GdlH} Ghys E., de la Harpe P., {\it Sur les groupes hyperboliques d'apr\`es Mikhael Gromov}, Progress in Mathematics, Vol. 83, Birkh\"auser, Boston 1990.

\bibitem[Gou1]{G} Gou\"ezel S., {\it Local limit theorem for symmetric random walks in Gromov-hyperbolic groups}, J. Amer. Math. Soc. Vol 27, No 3, 893-928 (2014).

\bibitem[Gou2]{Gpre} Gou\"ezel S., {\it Martin boundary of measures with infinite support in hyperbolic groups}, preprint, arXiv:1302.5388v1 [math.PR] (2013).

\bibitem[Gou3]{G2} Gou\"ezel S., Private communication, 2014.

\bibitem[GL]{GL} Gou\"ezel S., Lalley S P., {\it Random walks on co-compact Fuchsian groups}, Ann. Sci. \'Ec. Norm. Sup\'er., (4) {\bf 46} no. 1, 129-173 (2013).

\bibitem[GMM1]{GMM} Gou\"ezel S., Math\'eus F., Maucourant F., {\it Sharp lower bounds for the asymptotic entropy of symmetric random walks}, preprint, arXiv:1209.3378v3 [math.PR] (2014).

\bibitem[GMM2]{GMMpre} Gou\"ezel S., Math\'eus F., Maucourant F., {\it Entropy and drift in word hyperbolic groups}, arXiv:1501.05082v1 [math.PR] (2015).


\bibitem[Gui]{Gu} Guivarc'h Y., {\it Sur la loi des grands nombres et le rayon spectral d'une marche al\'eatoire}, Conference on Random Walks (Kleebach, 1979), 47-98, Ast\'erisque, {\bf 74}, 
Soc. Math. France, Paris, 1980.

\bibitem[Hai]{Ha} Haissinsky P., {\it Marches al\'eatoires sur les groupes hyperboliques}, G\'eom\'etrie Ergodique (Dal'Bo-Milonet, ed.), Monographie de L'Enseignement Math\'ematique, {\bf 43}, 199-265 (2013).

\bibitem[HMM]{HMM} Haissinsky P., Mathieu P., M\"uller S., {\it Renewal theory for random walks on surface groups}, arXiv:1304.7625v1 [math.PR] (2013).

\bibitem[Hei]{H} Heinonen J., {\it Lectures on Analysis on Metric Spaces}, Springer-Verlag, New York 2001.

\bibitem[INO]{INO} Izumi M., Neshveyev S., Okayasu R., {\it The ratio set of the harmonic measure of a random walk on a hyperbolic group}, Israel J. Math. {\bf 163}, 285-316 (2008).

\bibitem[Kai1]{Ktree} Kaimanovich V A., {\it Hausdorff dimension of the harmonic measure on trees}, Ergodic Theory Dynam. Systems, {\bf 18}, 631-660 (1998). 

\bibitem[Kai2]{K} Kaimanovich V A., {\it The Poisson formula for groups with hyperbolic properties}, Ann. Math. Second series, Vol. 152, No. 3, 659-692 (2000). 

\bibitem[KL]{KLeP} Kaimanovich V A., Le Prince V.,{\it Matrix random products with singular harmonic measure}, Geom. Dedicata {\bf 150} (2011) 257-279.

\bibitem[KV]{KV} Kaimanovich V A., Vershik A M., {\it Random walks on discrete groups: boundary and entropy}, Ann. Probab, Vol. 11, No. 3, 457-490 (1983). 

\bibitem[Led]{L} Ledrappier F., {\it Some asymptotic properties of random walks on free groups},  Topics in probability and Lie groups: boundary theory, 117-152 , CRM Proc. Lecture Notes, {\bf 28}, Amer. Math. Soc., Providence, RI, 2001.

\bibitem[LeP]{LeP} Le Prince V., {\it Dimensional properties of the harmonic measure for a random walk on a hyperbolic group}, Trans. Amer. Math. Soc., Vol. 359, No. 6, 2881-2898 (2007). 

\bibitem[Lyo]{Lyo} Lyons R., {\it Equivalence of boundary measures on covering trees of finite graphs}, Ergodic Theory Dynam. Systems, {\bf 14}, 575-597 (1994).

\bibitem[LP]{LP} Lyons R., Peres Y., {\it Probability on Trees and Networks}, Cambridge University Press. In preparation. Current version available at {\tt http://mypage.iu.edu/$\tilde{ \ }$rdlyons/}

\bibitem[LPP]{LPP} Lyons R., Pemantle R., Peres Y., {\it Ergodic theory on Galton-Watson trees: speed of random walk and dimension of harmonic measure}, Ergodic Theory Dynam. Systems, {\bf 15}, 593-619 (1995).

\bibitem[PP]{PP} Parry W., Pollicott M., {\it Zeta functions and the periodic orbit structure of hyperbolic dynamics}, Ast\'erisque, 187-188, Soci\'et\'e Math\'ematique de France (1990).

\bibitem[Pat]{P} Patterson S J., {\it The limit set of a Fuchsian group}, Acta Math., 136, 241-273 (1976).

\bibitem[Pes]{Pe} Pesin Y B., {\it Dimension Theory in Dynamical Systems}, The University of Chicago Press, Chicago and London (1997).

\bibitem[PU]{PU} Przytycki F., Urba\'nski M., {\it Conformal Fractals: Ergodic Theory Methods}, London Mathematical Society Lecture Note Series 371, Cambridge University Press, Cambridge 2010.

%\bibitem[S]{S} Series C., {\it The infinite word problem and limit sets in Fuchsian groups}, Ergodic Theory Dynam. Systems, {\bf 1}, 337-360 (1981).

\bibitem[Ver]{V} Vershik A M., {\it Dynamic theory of growth in groups: Entropy, boundaries, examples}, Russian Math. Surveys, {\bf 55}:4, 667-733 (2000).

\bibitem[Woe]{W} Woess W., {\it Random walks on infinite graphs and groups}, Cambridge Tracts in Mathematics, {\bf 138}, Cambridge University Press, Cambridge, 2000.





\end{thebibliography}
\end{document}